\documentclass[12pt]{article}
\usepackage{authblk}
\usepackage{geometry}                
\usepackage{graphicx}
\usepackage{amssymb,amsmath,amsthm,amsfonts,bbm,bm,mathrsfs,xcolor,mathtools}
\usepackage{mhequ}
\usepackage{mathabx}
\usepackage{comment}
\usepackage[authoryear]{natbib}
\usepackage[colorlinks,citecolor=blue,urlcolor=blue,filecolor=blue,backref=page]{hyperref}
\usepackage{accents,url}
\usepackage[colorinlistoftodos,prependcaption,textsize=tiny]{todonotes}
\usepackage{thmtools,thm-restate}
\usepackage{epstopdf}
\usepackage{multicol}

\usepackage[inline]{trackchanges}
\addeditor{lp}
\addeditor{ab}

\DeclareMathOperator*{\arginf}{arg\,inf}
\DeclareMathOperator{\TV}{TV} 

\DeclarePairedDelimiter\ceil{\lceil}{\rceil}

\def\mr{\mathrm}

\def\bx{\mathbf{x}}
\def\Dir{\text{Dirichlet}}
\def\Hist{\text{Histogram}}
\def\HistDens{\text{HistogramDensity}}
\def\Mult{\text{Multinomial}}
\def\Res{\text{Res}}
\def\Var{\text{Var}}
\def\Beta{\text{Beta}}
\def\split{\text{Split}}

\numberwithin{equation}{section}
\theoremstyle{plain}

\newtheorem{assumption}{Assumption}

\newtheorem{result}{Result}[section]
\newtheorem{lemma}{Lemma}

\newtheorem{definition}{Definition}

\title{Memory Efficient And Minimax Distribution Estimation Under Wasserstein Distance Using Bayesian Histograms}
\author[1,2]{Peter Matthew Jacobs}
\author[1,2]{Lekha Patel\footnote{Corresponding author: lpatel@sandia.gov.}}
\author[3]{Anirban Bhattacharya}
\author[3]{Debdeep Pati}
\affil[1]{Scientific Machine Learning, Sandia National Laboratories, Albuqueque, NM 87123, USA}
\affil[2]{Kahlert School of Computing, University of Utah, Salt Lake City, UT 84112, USA}
\affil[3]{Department of Statistics, Texas A\&M University, College Station, TX 77843, USA}
\date{\vspace{-1.5em}}                                           

\begin{document}
\maketitle

\begin{abstract}
    We study Bayesian histograms for distribution estimation on $[0,1]^d$ under the Wasserstein $W_v, 1 \leq v < \infty$ distance in the i.i.d sampling regime. We newly show that when $d < 2v$, histograms possess a special \textit{memory efficiency} property, whereby in reference to the sample size $n$, order $n^{d/2v}$ bins are needed to obtain minimax rate optimality. This result holds for the posterior mean histogram and with respect to posterior contraction: under the class of Borel probability measures and some classes of smooth densities. The attained memory footprint overcomes existing minimax optimal procedures by a polynomial factor in $n$; for example an $n^{1 - d/2v}$ factor reduction in the footprint when compared to the empirical measure, a minimax estimator in the Borel probability measure class. Additionally constructing both the posterior mean histogram and the posterior itself can be done super--linearly in $n$. Due to the popularity of the $W_1,W_2$ metrics and the coverage provided by the $d < 2v$ case, our results are of most practical interest in the $(d=1,v =1,2), (d=2,v=2), (d=3,v=2)$ settings and we provide simulations demonstrating the theory in several of these instances. 
\end{abstract}

\section{Introduction} \label{sec: intro}

The Wasserstein metric is a popular tool for comparing two distributions $\mu$ and $\nu$ defined on a common metric space $(E,\tilde{d})$. For $1 \leq v < \infty$, the Wasserstein distance $W_v$ is defined as 
\begin{equation}
    \label{wassersteinDef}
    W_v(\mu,\nu) := \left (\inf_{\pi \in \mathcal{M}(\mu,\nu)} \int \tilde{d}(x,y)^v \; \mr{d} \pi(x,y) \right )^{1/v},
\end{equation}
where $\mathcal{M}(\mu,\nu)$ is the set of couplings of $\mu$ and $\nu$; specifically the joint probability measures on $E \times E$ with marginals $\mu$ and $\nu$ respectively.
Some benefits of the using the Wasserstein metric include its sensitivity to the topology of the underlying space, ability to compare two measures  regardless of their levels of continuity, and its 1-dimension equivalent representation as the $L^v$ distance between quantile functions, which facilitates quantile function inference \citep{Zhang2020}. The Wasserstein metric is used in a variety of application settings, including a wealth of two and three dimensional problems inherent in image and video analysis \citep{rubner2000earth,sandler2011nonnegative,baumgartner2018visual,wu2021era}.

In this paper we study the problem of non--parametrically estimating a distribution on $[0,1]^d$ under the Wasserstein metric from $n$ independent and identically distributed (i.i.d) random variables $Y_1, \dots, Y_n$. As detailed in Section \ref{sec:priorWork}, this problem has recently received heightened attention. Within this framework, a particular focus lies on memory efficient estimation, which is  
critical in large data problems in which storing an $n$ atom measure, such as the \textit{empirical measure}, is not computationally feasible. Memory efficient distribution estimation under the Wasserstein distance is also important for its implications in time efficient and statistically accurate estimation of the Wasserstein distance itself. Two problems that illustrate this point are \textit{Minimum Wasserstein Distance Estimation (MWDE)} \citep{bernton2017inference,bernton2019parameter,bassetti2006minimum} and \textit{Approximate Bayesian Computation (ABC)} \citep{bernton2019approximate,legramanti2022concentration}. In MWDE, a parametric family of likelihoods $\mathcal{P}_{\theta} = \{f_\theta | \theta \in \Theta \}$ is to be fit to an unknown $P_0$, with the goal of selecting $\hat{\theta} = \arginf_{\theta \in \Theta} W_v(P_0,f_{\theta})$. After estimating $P_0$ with the empirical measure $\hat{P}_n := n^{-1}\sum_{i=1}^n \delta_{Y_i}$, where $\delta_a$ denotes a point mass at $a$, the actual optimization problem to be solved is $\hat{\theta} = \arginf_{\theta \in \Theta} W_v(\hat{P}_n,f_{\theta})$. However, as is discussed in \cite{bernton2017inference}, the Wasserstein distance $W_v(\hat{P}_n,f_{\theta})$ may frequently be computationally or analytically intractable for the parametric family under consideration. Instead, a common strategy is to minimize $k^{-1} \sum_{t=1}^{k} W_v(\hat{P}_n,\hat{f_{\theta}}_{m,t})$ over $\theta \in \Theta$ where $\hat{f_{\theta}}_{m,1},\dots,\hat{f_{\theta}}_{m,k}$ are $m$ sample empirical measures derived from samples of $f_{\theta}$ and $k$ is an integer chosen sufficiently large.  Optimization algorithms used in this setting (such as Nelder--Mead, used in \cite{bernton2019approximate}) rely on many Wasserstein distance evaluations. Existing procedures for exact or approximate \citep{cuturi2013sinkhorn,gottschlich2014shortlist,altschuler2017near,luise2018differential,peyre2019computational,chizat2020faster} computation of Wasserstein distance between discrete distributions have a runtime that depends polynomially on the number of atoms in the measures being compared. For example, in the $d \geq 2$ setting when $n = m$, the worst case runtime of exact Linear Programming based solvers is of order $n^3 \log(n)$ \citep[see][Section~2.1]{pele2009fast} while for the Sinkhorn solver it is of order nearly $n^2$ \citep[see][Chapter~4]{peyre2019computational}. Thus, finding alternatives to $\hat{P}_n$ (and $\hat{f}_{{\theta}_{m,t}}$) that have a smaller memory footprint but are still  estimators for $P_0$ (and $f_{\theta}$) with statistical quality, is important. Similarly, in Wasserstein ABC, approximate samples from a posterior distribution are generated by randomly sampling a $\theta$ from the prior distribution, generating an $m$ sample synthetic dataset (and associated empirical measure ${\hat{f}_{\theta}}_m$) from $f_\theta$, the likelihood associated with $\theta$, and then accepting $\theta$ as a posterior sample when $W_v(\hat{P}_n,\hat{f_{\theta}}_m)$ is sufficiently small. Here, having a low atom substitute for $\hat{P}_n$ (and for $\hat{f_{\theta}}_m$) is computationally important due to the vast number of $W_v$ computations that are required to be performed.

Conversely, there is also a need to supply practitioners with methods for assessing the uncertainty in distribution estimation. Although there are various ways of performing uncertainty quantification for non--parametric problems (see \cite{mcdonald2021review}), a measure of uncertainty given by a Bayesian approach is attractive as it is non-asymptotic and is a statement conditional on the observed data. A widely accepted way \citep{ghosal2017fundamentals} for assessing a Bayesian method in a frequentist sense is to find conditions under which a posterior assigns vanishingly small mass outside a shrinking ball around the truth. If the minimum radius of this shrinking ball (called the posterior contraction rate) matches with the classical minimax rate associated with the problem, we say the Bayesian method is agnostic to the prior choice in the presence of an infinite amount of data. Such a result provides some comfort to the practitioner about the possible effects of the prior choice in the underlying inference. Moreover, point estimates obtained as appropriate summaries of an optimally contracting posterior automatically provide minimax optimal frequentist estimators.
The work of \cite{ghosal2000} provides a general three condition strategy for proving these posterior contraction rates, yet may be more difficult to use for the Wasserstein metric. This occurs since $W_v, v \geq 2$ is not dominated by Total Variation (TV) or Hellinger distances, and because for $W_1$ (at least when considering classes of probability measures with densities that are H\"{o}lder smooth of regularity $0 < s \leq 1$), the minimax rate of convergence under the Kullback–Leibler (KL) divergence is slower than under $W_1$. These challenges generally make the task of satisfying the test construction and prior thickness conditions, while retaining the minimax rate, difficult.

In this paper, we consider the $([0,1]^d,\| \cdot - \cdot \|_p)_{d \in \mathbb{N}, 1 \leq p < \infty}$ metric space, where $\| \bx \|_p : = (\sum_{i=1}^n x_i^p)^{1/p}$ denotes the $L^p$ norm. On this space, we show that the Bayesian histogram model yields minimax optimal procedures for estimating distributions under the Wasserstein distance, under the class of Borel probability measures and some classes of smooth densities. Specifically, we study both the posterior mean histogram sequence as well as the rate of posterior contraction. Our main results show that when $d < 2v$, the posterior mean histogram and posterior distribution require storing only order $n^{d/2v}$ atoms (in reference to the sample size $n$) for minimax optimality to be achieved. By succinctly leveraging conjugacy in the Bayesian histogram, we demonstrate an alternative approach to the \cite{ghosal2000} method for proving a posterior contraction rate under Wasserstein distances. 

An important reason for studying the Bayesian histogram (as opposed to only the frequentist one) is that the prior can ensure Wasserstein distance computations using the histogram as an estimate for $P_0$ are possible. Specifically, semi-discrete Optimal Transport algorithms \citep{kitagawa2019convergence,merigot2011multiscale} are used to compute the Wasserstein distance between a discrete distribution and a density. The correctness of common algorithms in this space rely on the support of the density being connected. When $P_0$ is estimated with the exact posterior mean histogram density we present in this paper, a non-trivial prior will ensure the support of the estimator for $P_0$ is the entire domain, thereby facilitating Wasserstein distance calculations against discrete distributions. We revisit this point and discuss in more detail the different options the practitioner has for using the histogram, or a discretization of the histogram, as an estimator for $P_0$ in computation of Wasserstein distances in Section \ref{sec:memEff}.


The paper is organized as follows. Section \ref{sec:priorWork} provides a summary of prior work on similar problems. Section \ref{sec:notation} describes notational conventions. Section \ref{sec:theory} details the Bayesian histogram model and states the main theorems on rates of convergence and discusses these results. Section \ref{sec:memEff} discusses applications of the theory in more detail. Section \ref{sec:sims} provides simulations to demonstrate our theoretical results. Finally, Section \ref{sec:conclusion} provides some concluding remarks. Proofs of the main theorems, including intermediate results, can be found in the Appendix.

\subsection{Prior work}
\label{sec:priorWork}

The frequentist convergence rates of the empirical measure under the expected Wasserstein distance are studied in \cite{Fournier2015,singh2018minimax, bobkov2019one,weed2019sharp} to varying degrees of generality. A consequence of the work of \cite{singh2018minimax} is that on the metric space $([0,1]^d,\| \cdot - \cdot \|_p)$, for $d \in \mathbb{N}, 1 \leq p < \infty$, for the class of Borel probability measures, the empirical measure is minimax optimal (at least up to logarithmic terms) for every $v \geq 1$. Further, the minimax rate is lower bounded by $n^{-1/2v}$ for $d \leq 2v$, and $n^{-1/d}$ for $d  > 2v$.

\cite{niles2022minimax} study minimax rates for the Wasserstein distance under classes of smooth densities. Specifically, for $s,\tilde{L} >0, 1 \leq p',q \leq \infty$, and letting
\begin{equation}
    \label{besovSpaceDefinition}
    B_{p',q}^{s}(\tilde{L}) = \left\{f \in L_{p'}([0,1]^d): \| f - \pmb{1} \|_{B_{p',q}^s} \leq \tilde{L}, \int f =1, f \geq 0 \right\},
\end{equation}
where $B_{p',q}^{s}$ is the Besov norm, \citet[~Theorem 5]{niles2022minimax} show that for $p', q, v \geq 1$, $s > 0$ and some $\tilde{L}$ sufficiently large, the minimax lower bound follows 
\begin{equation}
\label{smoothnessClass}
\inf_{\tilde{\mu}} \sup_{f \in B^{s}_{p',q}(\tilde{L})} W_v(f, \tilde{\mu}) \gtrsim \begin{cases} n^{-\frac{1+\frac{s}{v}}{d+s}} & d - s \geq 2v \\ n^{-\frac{1}{2v}} & d - s < 2v, \end{cases} 
\end{equation}
where infimum is over estimators $\tilde{\mu}$ based on $n$ observations and $\gtrsim$ is in reference to $n$ (specifically constants are allowed to depend on quantities not depending on $n$).
In particular, we will use this result to argue that the Bayesian memory efficient histogram studied in this paper retains minimax optimality in the class $B^{s}_{p',q}$ for $p', q, v \geq 1$, provided $d < 2v$ for any $s > 0$. In the H\"{o}lder class $\mathcal{C}^{s} = B_{p' = \infty, q = \infty}^{s}$ \citet[~Theorem 6]{niles2022minimax} introduce a histogram estimator that is minimax optimal up to logarithmic terms when $v \geq 2$ and $0 < s < 1$. While this estimator is defined using the Haar wavelet basis, we prove in Appendix Section \ref{app:nwbConnect} that this histogram belongs to the dyadic histogram class discussed in this paper, but with Dirichlet weights of zero prior concentration and where there are order $n^{d/(d+s)}$ bins total. While in the \hbox{$d < 2v$} case construction of the Bayesian memory efficient histogram does not depend on the regularity \hbox{$s$} and still the minimax lower bound of Equation \ref{smoothnessClass} is achieved, the construction of the histogram in \cite{niles2022minimax} depends on knowledge of the regularity, which in practice may not be available. Moreover, in the most practically important $d < 2v$ cases, $(d=1,v =1,2), (d=2,v=2), (d=3,v=2)$, $n^{d/2v}$ is a polynomial factor smaller than $ n^{d/(d+s)}$ for $s < 1$. Therefore, even if the regularity is known prior to data collection, the memory footprint of the histogram constructed in \cite{niles2022minimax} never outperforms that of the Bayesian histogram studied here. Finally, we note that the proofs of minimax optimality for the estimators presented in this paper are not restricted to the H\"{o}lder classes.

The frequentist and posterior contraction proofs that we show rely on the \textit{multiresolution} upper bound for the Wasserstein distance. This follows several other works such as \cite{Fournier2015, weed2019sharp, singh2018minimax, niles2022minimax}, that utilize a version of the multiresolution upper bound on the Wasserstein distance.

Regarding posterior contraction rates for Bayesian histograms, \cite{scricciolo2007rates} uses the standard \cite{ghosal2000} strategy to provide minimax optimal posterior contraction for the class of $0 < s \leq 1$ H\"{o}lder continuous densities on $[0,1]$ and $[0,1]^2$, under the Total Variation ($\TV$) and Hellinger metrics. Moreover, \cite{castillo2014bernstein} demonstrate that asymptotically, posterior probabilities of Kolmogorov-Smirnov balls around $\hat{P}_n$ agree with frequentist ones when estimating $s$ H\"{o}lder continuous densities on $[0,1]$ where $1/2 < s \leq 1$. For metrics that are not dominated by $\TV$ or Hellinger, such as for $W_v, v \geq 2$, explicit test construction is necessary within the realm of the \cite{ghosal2000} strategy. In non-histogram models and for non-Wasserstein metrics, explicit test construction is carried out in \cite{pati2014posterior} and \cite{gine2011rates}. In the problem of inferring the mixing measure of a mixture density, posterior contraction rates under the Wasserstein distance have been studied. Specifically \cite{nguyen2013convergence} and \cite{gao2016posterior} upper bound powers of the Wasserstein distance between mixing measures by the Hellinger distance between the respective mixing densities. \cite{gao2016posterior} develops a generalization of the three condition theorem of \cite{ghosal2000} to prove posterior contraction for the powered Wasserstein measurement (a non-distance) between distributions, which then implies a rate of posterior contraction under Wasserstein distance.

In parametric problems using Euclidean distance to quantify loss, leveraging conjugacy where possible, is a common strategy for analyzing the posterior \citep{van2014horseshoe}. In the non-parametric setting, conjugacy has also been exploited to study posterior asymptotics, with examples including inferring the mean function in a regression employing a Gaussian process prior on the mean \citep{knapik2011bayesian,yang2017frequentist}.
\section{Notation and definitions}
\label{sec:notation}
We consider the metric space $([0,1)^d,\| \cdot - \cdot \|_{p})$ where $d \in \mathbb{N}$ and $1 \leq p < \infty$. Recall that for $p \geq 1$, if $\pmb{x} = (x_1,\dots,x_d),\pmb{y} = (y_1,\dots,y_d) \in [0,1)^d$, then
\begin{equation}
    \|\pmb{x} - \pmb{y}\|_{p} := \left( \sum_{i=1}^{d} |x_i-y_i|^p \right)^{1/p}.
\end{equation}
The corresponding space of probability measures considered is  
\begin{equation}
    \mathcal{P}_d := \{ \text{Borel Probability Measures on } [0,1)^d \}. \label{def: prob space}
\end{equation}
Excluding the right end points are a notational convenience but extension of the arguments that follow to include the right endpoint is trivial. For $v \in [1,\infty)$, we study the Wasserstein-$v$ distance between two probability measures $\mu,\nu \in \mathcal{P}_d$ where $\tilde{d}$ of Equation \ref{wassersteinDef} satisfies $\tilde{d} = \| \cdot - \cdot \|_p$.
Note that $p$ is suppressed in the notation $W_v(\mu,\nu)$. 

To address the notational conventions that follow in this paper, we first note that the empirical measure based on $n$ i.i.d samples $Y_1,\dots,Y_n$ is denoted $\hat{P}_n$, with $\hat{P}_n = n^{-1} \sum_{i=1}^n \delta_{Y_i}$. $\log$ without a base explicitly given refers to the natural logarithm. $a_n \lesssim r_n$ means that there exists a $C >0$ not depending on $n$ and $N \in \mathbb{N}$ such that for $n \geq N$, $a_n \leq C r_n$. $C$ may depend on $d$ in this work and we view this as reasonable because the memory efficiency gains occur when $d < 2v$ and $v=1,2$ are most often used in practice. Further $a_n \lesssim r_n$ and $r_n \lesssim a_n$ is denoted $a_n \asymp r_n$. Finally, note that $i.p$ means \textit{in probability}, and $\mathbb{I}(\cdot)$ denotes the standard indicator function.

For $b,d \in \mathbb{N}$, we denote $[b] := \{1,2,\dots,b\}$ and $[b]^d := \prod_{j=1}^{d} [b]$.
For $j \in \mathbb{N}$, $\mathcal{S}^{j-1}$ refers to the $(j-1)$ dimensional probability simplex. That is $\mathcal{S}^{j-1} := \{ (x_1,\dots,x_j) \in \mathbb{R}^{j}: \sum_{t=1}^{j} x_t =1, x_t \geq 0 \text{ for } t \in (1,2,\dots,j) \}$. Also note that $\mathbb{R}_{+} := \{ x \in \mathbb{R}: x > 0 \}$ and for $z \in \mathbb{N}$ and $\pmb{\alpha} \in \mathbb{R}_{+}^{z}$, the Dirichlet probability measure $\Dir: \mathcal{B}(\mathcal{S}^{z-1}) \to [0,1]$ is given by 
\begin{equation}
    \Dir(G|\pmb{\alpha}) = \frac{1}{B(\pmb{\alpha})} \int_{G} \prod_{i=1}^{z} x_i^{\alpha_i -1} \mr{d} \pmb{x},
\end{equation}
where $B$ is the $z$ dimensional Beta function and where $\mathcal{B}(\mathcal{S}^{z-1})$ is the Borel measurable subsets of $\mathcal{S}^{z-1}$ and $G \in \mathcal{B}(\mathcal{S}^{z-1})$. For $b \in \mathbb{N}$ and a multi-index $\pmb{i} = (i_1,i_2,\dots,i_d) \in [b]^d$, define
\begin{equation}
\label{AsetDefinition}
A_{\pmb{i},b} := \left[\frac{i_1-1}{b},\frac{i_1}{b} \right) \times \left[\frac{i_2-1}{b},\frac{i_2}{b}\right) \times \dots, \times \left[\frac{i_d-1}{b},\frac{i_d}{b}\right).
\end{equation}
Clearly, $\{A_{\pmb{i},b}\}_{\pmb{i} \in [b]^d}$ form a partition of $[0, 1)^d$.
For a vector of weights $\pmb{\pi} = \{ \pi_{\pmb{j}} \}_{\pmb{j} \in [b]^d} \in \mathcal{S}^{bd-1}$, the $d$ dimensional Histogram probability measure $\Hist: \mathcal{B}([0,1)^d) \to [0,1]$ is a weighted mixture of uniform distributions on the partition sets $A_{\pmb{i},b}$, defined by 
\begin{equation}
\Hist(G|\pmb{\pi},b) := \int_{G} \sum_{\pmb{i} \in [b]^d} b^{d} \pi_{\pmb{i}} \mathbb{I}(\pmb{y} \in A_{\pmb{i},b}) \mr{d} \pmb{y},
\end{equation}
where $\mathcal{B}([0,1)^d)$ is the Borel measurable subsets of $[0,1)^d$ and $G \in \mathcal{B}([0,1)^d)$.
\section{d-dimensional Bayes histogram}
\label{sec:theory}
In this section, we define the \hbox{$d$}-dimensional Bayesian histogram, present the main theorems giving upper bounds on expected loss under Wasserstein distances for the posterior mean histogram and a rate of contraction for the posterior itself, and discuss the main technical tools used in the proofs. Discussion of the main results concludes this section.
\subsection{Setup}
 We suppose $Y_1,Y_2,\dots,Y_n,\dots \overset{iid}{\sim} P_0$ where $P_0 \in \mathcal{P}_d$.
For $b \in \mathbb{N}$, let $\pmb{\alpha}_b := \{ \alpha_{\pmb{j},b} \}_{\pmb{j} \in [b]^d} \in \mathbb{R}_+^{bd}$. 
For an increasing sequence $k_n$, let $b_n := 2^{\ceil{\log_2(k_n)}}$, $\pmb{\pi}_n := \{ \pi_{n,\pmb{j}} \}_{\pmb{j} \in [b_n]^d} \in \mathcal{S}^{b_n d-1}$. For $n \in \mathbb{N}$, the Bayesian Histogram model likelihood and prior are given by

\begin{align}
\label{modelDefinition}
Y_1,\dots,Y_n | \pmb{\pi}_n \overset{i.i.d}{\sim} \Hist(\cdot|\pmb{\pi}_n,b_n) & \qquad &
\pmb{\pi}_n | \pmb{\alpha}_{b_n} \sim \Dir(\cdot | \pmb{\alpha}_{b_n}).
\end{align}
Also, let $z_n^{*}(\cdot | Y_1,\dots,Y_n)$ refer to the posterior probability measure over $\mathcal{S}^{b_n d-1}$ derived from Equation \ref{modelDefinition}. 

As $\alpha_{\pmb{i},b_n} >0$ for every $\pmb{i} \in [b_n]^d$ and for every $n \in \mathbb{N}$, Equation \ref{modelDefinition} induces a sequence of posterior distributions over $\mathcal{P}_d$. Specifically let $\psi_b: \mathcal{S}^{bd-1} \to \mathcal{P}_d$ be the map that takes a given $\pmb{\pi} = \{\pi_{\pmb{j}} \}_{\pmb{j} \in [b]^d }$ and produces its corresponding histogram probability measure. That is 
\begin{equation}
    \psi_b(\pmb{\pi}) = \Hist(\cdot|\pmb{\pi},b).
\end{equation}
For a measurable set $B \subseteq \mathcal{P}_d$, the posterior measure $\Pi_n$ is 
\begin{equation}
\label{preImageForm}
\Pi_n(B|Y_1,\dots,Y_n) = z_n^{*}(\psi_{b_n}^{-1}(B)|Y_1,\dots,Y_n). 
\end{equation}

Due to conjugacy, it is straightforward to show that $$z_n^*(\cdot|Y_1,\dots,Y_n) = \Dir(\cdot| \pmb{\alpha_{k_n}^{*}}),$$ where for $\pmb{i} \in [b_n]^d$
\begin{equation}
\alpha_{\pmb{i},b_n}^{*}= \alpha_{\pmb{i},b_n} + \sum_{j=1}^{n} \mathbb{I}(Y_j \in A_{\pmb{i},b_n}).
\end{equation}
Now allowing $\pmb{\alpha}_{b_n} \in \{x \in \mathbb{R}: x \geq 0\}^{b_n d}$, we define the sequence of estimators for $P_0$, denoted $\bar{P}_n$, by
\begin{equation}
\label{posteriorMeanHistogram}
\bar{P}_n := \psi_{b_n}\left\{\left(\frac{\alpha_{\pmb{i},b_n}^{*}}{\sum_{\pmb{j} \in [b_n]^d} \alpha_{\pmb{j},b_n}^{*}}\right)_{\pmb{i} \in [b_n]^d} \right\} = \psi_{b_n}\{ ( E_{z_n^*}(\pi_{\pmb{i}} | Y_1,\dots,Y_n) )_{\pmb{i} \in [b_n]^d} \},
\end{equation}
where the second equality above holds if $\alpha_{\pmb{i},b_n} >0$ for $\pmb{i} \in [b_n]^d$, and $\bar{p}_n$ denotes the density associated with $\bar{P}_n$. 

We note that posterior distributions derived from improper prior distributions are not considered in this work, and therefore to consider the posterior measure sequence $\Pi_n$ we require that $\alpha_{\pmb{i},b_n} >0$ for $\pmb{i} \in [b_n]^d$. However, we allow $\bar{P}_n$ to be defined regardless of whether or not the prior distribution over the simplex is proper. In particular, it is still defined in the event that some or all of the $\alpha_{\pmb{i},b_n}$ parameters are zero. When the prior distribution is proper, $\bar{P}_n$ has an additional interpretation: it is the posterior mean histogram. 
In the lemmas and theorems that follow that involve analysis of the posterior distribution sequence $\Pi_n$, we make clear that we require $\alpha_{\pmb{i},b_n} >0$ for $\pmb{i} \in [b_n]^d$ and $n \in \mathbb{N}$.

$\bar{P}_n$ (and $\Pi_n$) refer to an entire class of point estimators (and posterior distributions) indexed by the parameter sequences $k_n$ and $\pmb{\alpha}_{b_n}$. In the subsequent subsection we establish constraints on $k_n$ and $\pmb{\alpha}_{b_n}$ that ensure $\bar{P}_n$ and $\Pi_n$ are minimax statistical procedures in certain general distribution classes while still maintaining memory efficiency in the $d < 2v$ case. From herein, denote $K_n := \ceil{\log_2(k_n)}$, so that $b_n = 2^{\ceil{\log_2(k_n)}} = 2^{K_n}$.

\subsection{Analysis}
All results that follow will apply for every $p \geq 1$. Statements are asymptotic in $n$ but \textit{not} in $d$. The scaling factors $C(d,v)$ that we define appearing in the posterior contraction theorems, simply make the role of $d$ and $v$ explicit. 

Our analysis does not impose any smoothness assumption on $P_0$. We study both the rate at which the expected Wasserstein distance between $P_0$ and $\bar{P}_n$ decays, as well as the contraction rate of $\Pi_n$ under Wasserstein neighborhoods of $P_0$.   According to \cite{singh2018minimax},
\begin{equation}
    \label{minimaxBorelProbMeasures}
    \inf_{\Tilde{P}} \sup_{P_0 \in \mathcal{P}} \mathbb{E}_{P_0} W_v(\Tilde{P},P_0) \gtrsim \begin{cases}
n^{-\frac{1}{2v}} & d \leq 2v, \\
n^{-\frac{1}{d}} & d > 2v,
\end{cases}
\end{equation}
where the $\inf$ is taken over all estimators $\tilde{P}$ from $n$ observations. Thus these are the rates we aim for (and achieve at least up to logarithmic terms) in the subsequent theorems. Theorem \ref{frequentistDwasser} concerns $\mathbb{E}_{P_0} W_v(P_0,\bar{P}_n)$, while Theorem \ref{contractionAroundCentralEstimatorDwasser} establishes posterior contraction around $\bar{P}_n$. Theorems \ref{frequentistDwasser} and \ref{contractionAroundCentralEstimatorDwasser} are used to prove Theorem \ref{concludingTheoremDwasser}, which establishes posterior contraction around $P_0$. The proofs are given in Appendix Section \ref{app:mainTheoremsProofs}. After presenting the main theorems we mention the technical tools used to prove them.

Before stating the theorems, we define the following assumptions for $d \in \mathbb{N}$ and $v \geq 1$.

\begin{assumption}
For $n \in \mathbb{N}$
\label{knDef}
\[k_n = \begin{cases} n^{1/2v} & d \leq 2v, \\
                        n^{1/d} & d > 2v,
            \end{cases}
\]
\end{assumption}
and
\begin{assumption}
\label{alphaUpperBound}
    \[ \sum_{\pmb{j} \in [b_n]^d} \alpha_{\pmb{j},b_n} \lesssim 
    \begin{cases}
        n^{1/2} & d \leq 2v. \\
        n^{1-\frac{v}{d}} & d > 2v.
    \end{cases}
    \]
\end{assumption}

We are now ready to state the main theorems. First regarding $\mathbb{E}_{P_0} W_v(P_0,\bar{P}_n)$, we have the following result.

\begin{restatable}{theorem}{frequentistDwasser}
\label{frequentistDwasser}
Let $Y_1, \dots, Y_n \overset{iid}{\sim} P_0 \in \mathcal{P}_d$.
Suppose $k_n$ satisfies Assumption \ref{knDef}, $\pmb{\alpha}_{b_n}$ satisfies Assumption \ref{alphaUpperBound} and that for $n \in \mathbb{N}$ and $\pmb{j} \in [b_n]^d$, $\alpha_{\pmb{j},b_n} \geq 0$.
Then
\[
\mathbb{E}_{P_0} W_v(P_0,\bar{P}_n) \lesssim \begin{cases}
n^{-\frac{1}{2v}} & d < 2v. \\
n^{-\frac{1}{2v}} \log^{\frac{1}{v}}(n) & d = 2v. \\
n^{-\frac{1}{d}} & d > 2v.
\end{cases}
\]
\end{restatable}

Second, regarding posterior contraction around $\bar{P}_n$, which serves as an intermediate step in proving posterior contraction around $P_0$, we have the following result. 

\begin{restatable}{theorem}{contractionAroundCentralEstimatorDwasser}
\label{contractionAroundCentralEstimatorDwasser}
Let $Y_1, \dots, Y_n \overset{iid}{\sim} P_0 \in \mathcal{P}_d$. Suppose $k_n$ satisfies Assumption \ref{knDef}.
Let $\gamma > 1$ and let $\{\tau_n(d,v)\}_{n = 1}^{\infty}$ be a sequence satisfying
\begin{equation}
\tau_n(d,v) =  \begin{cases}
 C_4(d,v)  n^{-\frac{1}{2v}} \log^{\frac{\gamma}{v}} (n)& d < 2v, \\
 C_5(d,v)  n^{-\frac{1}{2v}} \log^{\frac{1+\gamma}{v}}(n)	& d = 2v,\\
 C_6(d,v)  n^{-\frac{1}{d}} \log^{\frac{\gamma}{v}}(n)	& d > 2v,
 \end{cases}
\end{equation}
where 
\begin{align*}
C_4(d,v) &> \frac{1}{2} C_1(d,v)^{1/v} 2^{1/v} d^{1/p} &
C_5(d,v) &> \frac{1}{2} C_2(d,v)^{1/v} 2^{1/v} d^{1/p} \\ C_6(d,v) & > \frac{1}{2} C_3(d,v)^{1/v} 2^{1/v} d^{1/p},
\end{align*}
and
\begin{align*}
    C_1(d,v) &\geq \frac{2^{(\frac{d}{2}-v)}}{1-2^{(\frac{d}{2}-v)}} & C_2(d,v) & \geq \frac{2}{d\log(2)} & C_3(d,v) & \geq \frac{2^{2(\frac{d}{2}-v)}}{2^{(\frac{d}{2}-v)}-1}.
\end{align*}
Then, provided that $\alpha_{\pmb{j},b_n} > 0$ for each $n \in \mathbb{N}$ and $\pmb{j} \in [b_n]^d$, we have that for $1 \leq v < \infty$ and $d \in \mathbb{N}$
\[\mathbb{E}_{p_0} \Pi_n(P \in \mathcal{P}_d: W_v(P,\bar{P}_n) \geq \tau_n(d,v)) \to 0 \text{ as } n \to \infty.
\]
\end{restatable}

While Theorem \ref{frequentistDwasser} is important in its own right, Theorems \ref{frequentistDwasser} and \ref{contractionAroundCentralEstimatorDwasser} can be used to provide a posterior contraction rate for neighborhoods of $P_0$. Specifically, via multiple applications of Markov's inequality, we are able to construct Theorem \ref{concludingTheoremDwasser}.

\begin{restatable}{theorem}{concludingTheoremDwasser}
\label{concludingTheoremDwasser}
Let $Y_1, \dots, Y_n \overset{iid}{\sim} P_0 \in \mathcal{P}_d$. Suppose $\gamma > 1$ and $k_n$ satisfies Assumption \ref{knDef}
and
\begin{equation}
\epsilon_n(d,v) := 
\begin{cases}
C_7(d,v) n^{-\frac{1}{2v}} \log^{\frac{\gamma}{v}}(n)& d < 2v, \\
C_8(d,v) n^{-\frac{1}{2v}} \log^{\frac{1+\gamma}{v}} (n) & d = 2v, \\
C_9(d,v) n^{-\frac{1}{d}} \log^{\frac{\gamma}{v}} (n)& d > 2v,
\end{cases}
\end{equation}
where $C_7(d,v) \geq 2 C_4(d,v)$, $C_8(d,v) \geq 2 C_5(d,v)$ and $C_9(d,v) \geq 2 C_6(d,v)$. 
\sloppy Here $C_4(d,v),C_5(d,v),C_6(d,v)$ are set as in the statement of Theorem \ref{contractionAroundCentralEstimatorDwasser}. Now assuming that for each $n \in \mathbb{N}$ and $\pmb{j} \in [b_n]^d$, $\alpha_{\pmb{j},b_n} > 0$, and that $\pmb{\alpha}_{b_n}$ satisfies Assumption \ref{alphaUpperBound}
we have that for $1 \leq v < \infty$ and $d \in \mathbb{N}$
\[
\Pi_n(P \in \mathcal{P}_d : W_v(P_0,P) \geq \epsilon_n(d,v) ) \overset{i.p \; P_0}{\to} 0.
\]
\end{restatable}


There are three main technical tools (aside from conjugacy) used in the proofs of the main results. The first is Lemma 6 of \cite{singh2018minimax} which is a multiresolution upper bound on $W_v$. To state this lemma, we need the following two definitions. Definition \ref{def: nested partition} describes a nested partition, and Definition \ref{def: partition resolution} defines its resolution.

\begin{definition}[Nested Sequence of Partitions \citep{singh2018minimax} on page 20]
 \label{def: nested partition}
If $\mathcal{S},\mathcal{T}$ are partitions of a sample space $\Omega$, then $\mathcal{T}$ is a \textit{refinement} of $\mathcal{S}$ if for every $T \in \mathcal{T}$, there exists an $S \in \mathcal{S}$ with $T \subseteq S$. And a sequence $\{\mathcal{S}_k\}_{k=0}^{K}$ of partitions is called \textit{nested} if, for each $k \leq K-1$, $\mathcal{S}_{k+1}$ is a refinement of $\mathcal{S}_k$.
\end{definition}

\begin{definition}[Resolution of a Partition \citep{singh2018minimax}] \label{def: partition resolution}
If $\mathcal{S}$ is a finitely sized partition of the metric space $(\Omega,\rho)$, then $\Res(\mathcal{S}) = \max_{S \in \mathcal{S}} diam(S)$ where diameter is computed under $\rho$.
\end{definition}

Given these definitions, we now state the first technical tool in Result \ref{wassersteinMultiRes}, which gives the multiresolution upper bound of the Wasserstein distance. This bound is fundamental to the proofs of Theorems \ref{frequentistDwasser} and \ref{contractionAroundCentralEstimatorDwasser}.
\begin{result}[Wasserstein Multiresolution Upper Bound (Adapted Lemma 6 of \cite{singh2018minimax})]
\label{wassersteinMultiRes}
Let $\Omega \subseteq \mathbb{R}^d$ such that $diam(\Omega) < \infty$. For $K$ a positive integer, let $\{S_k\}_{k \in [K]}$ is a sequence of nested partitions of $(\Omega,\rho)$ with $\mathcal{S}_0 = \Omega$ where each partition has only finitely many elements. Then for any $1 \leq v < \infty$ and $\mu,\nu$ probability measures on $\Omega$
\begin{equation}
W_v^{v}(\mu,\nu) \leq \Res(\mathcal{S}_K)^v + \sum_{k=1}^{K} \left(\Res(\mathcal{S}_{k-1}) \right)^{v} \left(\sum_{S \in \mathcal{S}_k} | P(S) - Q(S) | \right).
\end{equation}
\end{result}

  \cite{singh2018minimax} prove a more general version of Result \ref{wassersteinMultiRes} in their work. There the partitions may have countably many elements and $\Omega$ is not necessarily a subset of $\mathbb{R}^{d}$. Here, however, we have only stated the lemma to the level of generality needed in our work.


The second technical tool is an upper bound on a Multinomial random variables expected distance from its mean under the distance induced by the $\| \cdot \|_1$ norm. This is used in the proof that upper bounds the expected loss of $\bar{P}_n$ (Theorem \ref{frequentistDwasser}). Specifically, Result \ref{multinomialConcentration} restates Lemma 8 of \cite{singh2018minimax}.
\begin{result}[Multinomial concentration (Lemma 8 of \cite{singh2018minimax})]
\label{multinomialConcentration}
\sloppy If $(X_1,\dots,X_k) \sim \Mult(n,p_1,\dots,p_k)$ and $Z := \sum_{j=1}^{k} |X_j - n p_j|$, then 
\[\mathbb{E} (Z/n) \leq \sqrt{\frac{k-1}{n}}.\]
\end{result}

The last technical tool is the Dirichlet distribution concentration around its mean in the distance induced by the $\|\cdot\|_1$ norm. We prove a sufficient concentration result for our purposes in Section \ref{diricheletConcentration}. See Appendix Section \ref{app:dConcentration} for the short proof of this result. Result \ref{diricheletConcentration} will be used in the proof giving a posterior contraction rate around $\bar{P}_n$ in Theorem \ref{contractionAroundCentralEstimatorDwasser}.
\begin{restatable}{result}{diricheletConcentration}
\label{diricheletConcentration} 
Let $k \in \mathbb{N}$ and $(\pi_1,\pi_2,\dots,\pi_k) \sim \Dir(\alpha_1,\alpha_2,\dots,\alpha_k)$. Then for $\delta > 0$
\[
\mathbb{P} \left(\sum_{j=1}^{k} |\pi_j - \mathbb{E}(\pi_j) | \geq  \frac{(\bar{\alpha})^{-\frac{1}{2}} \sqrt{k}}{\delta} \right) \leq \delta,
\]
where $\bar{\alpha} := \sum_{j=1}^{k} \alpha_j$.
\end{restatable}

For posterior contraction around $\bar{P}_n$, we use Result \ref{diricheletConcentration} to show that almost surely under $P_0$, the posterior Dirichlet probability that the $\|\cdot\|_1$ norm based difference between the $\pi$ parameters of the posterior and the $\pi$ parameters of $\bar{P}_n$ is smaller than the desired rate is tending to 1 as $n \to \infty$.

\subsection{Discussion of main results}
In this section, we discuss our main results with respect to general minimax and memory efficient inference theory, comparing the Bayesian histogram to other procedures for estimating distributions under Wasserstein distance. 
Throughout the remainder of the paper, references to $\bar{P}_n$ assume the sequences $k_n$ and $\pmb{\alpha}_{b_n}$ satisfy the conditions of Theorem \ref{frequentistDwasser} and references to $\Pi_n$ assume $k_n$ and $\pmb{\alpha}_{b_n}$ satisfy the conditions of Theorem \ref{concludingTheoremDwasser}.
\subsubsection{Minimax theory}
According to \cite{singh2018minimax},
\begin{equation}
    \label{minimaxBorelProbMeasures}
    \inf_{\Tilde{P}} \sup_{P_0 \in \mathcal{P}_d} \mathbb{E}_{P_0} W_v(\Tilde{P},P_0) \gtrsim \begin{cases}
n^{-\frac{1}{2v}} & d \leq 2v, \\
n^{-\frac{1}{d}} & d > 2v,
\end{cases}
\end{equation}
where the $\inf$ is taken over all estimators $\tilde{P}$ from $n$ observations. Thus by Theorem \ref{frequentistDwasser}, in the class $\mathcal{P}_d$, $\bar{P}_n$ provides a matching upper bound in the case that $d < 2v$ or $d > 2v$. In the case $d = 2v$, $\bar{P}_n$ matches the lower bound up to a logarithmic factor. The logarithmic factor mismatch with the minimax lower bound of Equation \ref{minimaxBorelProbMeasures} is expected in this case. This is because aside from perturbations caused by prior concentrations, $\bar{P}_n$ agrees with the empirical measure on members of partitions up to the resolution of the model, which is the deepest resolution considered during application of the Wasserstein multiresolution upper bound in Theorem \ref{frequentistDwasser}. Existing works that analyze the rate of convergence for the empirical measure (such as \cite{singh2018minimax} and \cite{ajtai1984optimal}) using a multiple partition based analysis do not avoid a logarithmic factor mismatch with the available minimax lower bound.

Moreover, according to \cite{singh2018minimax}, in all cases $d$ and $v$, these are the same rates proved for the empirical measure when considering class $\mathcal{P}_d$.
The contraction rate of Theorem \ref{concludingTheoremDwasser} matches the minimax lower bound, up to a logarithmic factor in all cases $d$ and $v$.

As discussed in Section \ref{sec: intro}, \cite{niles2022minimax} study Besov classes of densities $B_{p',q}^{s}(\tilde{L})$ (see Equation \ref{besovSpaceDefinition}). In particular, they show that if $\tilde{L} >0$ is sufficiently large, then for $1 \leq v < \infty, 1 \leq p',q \leq \infty, s >0$
\begin{equation}
    \label{minimaxBesov}
    \inf_{\Tilde{P}} \sup_{f_0 \in B_{p',q}^{s}(\tilde{L})} \mathbb{E}_{f_0} W_v(\Tilde{P},f_0) \gtrsim \begin{cases} n^{-\frac{1+s/v}{d+s}} & d-s \geq 2v, \\ 
    n^{-1/2v} & d-s < 2v, \end{cases}
\end{equation}
where again the $\inf$ is taken over all estimators $\tilde{P}$ from $n$ observations. Observe that when $d \leq 2v$, the regularity $s$ does not quicken the minimax rate relative to when considering the larger class $\mathcal{P}_d$. Thus, even though our proofs do not assume a density necessarily exists, we still have that when considering any $B^{s}_{p',q}(\tilde{L})$ class with sufficiently large $\tilde{L}$  where $1 \leq p',q \leq \infty, s >0$ and $d \leq 2v$, $\bar{P}_n$ achieves the minimax lower bound, and the contraction rate of Theorem \ref{concludingTheoremDwasser} achieves the lower bound up to logarithmic terms. 

On the other hand, consider the case $d > 2v$. If $s$ satisfies $d -s < 2v$, then the $n^{-1/2v}$ lower bound rate of Equation \ref{minimaxBesov} is polynomially decaying faster than the $n^{-1/d}$ upper bound we have proved for $\bar{P}_n$ and $\Pi_n$ in Theorems \ref{frequentistDwasser} and \ref{concludingTheoremDwasser} respectively. Further, if $d-s \geq 2v$, then again the $n^{-(1+s/v)/(d+s)}$ lower bound rate of Equation \ref{minimaxBesov} is polynomially decaying faster than the $n^{-1/d}$ upper bound we have proven for $\bar{P}_n$ and $\Pi_n$. 
As our proofs do not assume existence of an $s$ regular density, we now pose the question: when $d > 2v$, by assuming the data generating measure $P_0$ has a smooth density, can $\bar{P}_n$ and $\Pi_n$ achieve a rate decaying faster than $n^{-1/d}$?

In Appendix Section \ref{app:nwbConnect}, we show that the estimator presented in \citet[Theorem~6]{niles2022minimax} is a slight alteration to $\bar{P}_n$ where $\alpha_{i,b_n} = 0,i \in [b_n]^d$ and instead of setting $k_n$ as we do, they set $k_n \asymp n^{1/(d+s)}$ in all cases $d$ and $v$. This ensures that the number of bins in the histogram depends on the regularity $s$. Letting $\ddot{P}_n$ denote their estimator, it is shown that up to a logarithmic term, $\ddot{P}_n$ obtains the minimax lower bound of Equation \ref{minimaxBesov} when considering the \hbox{H\"{o}lder} continuous density classes $\mathcal{C}^{s}(\tilde{L}) = \mathcal{B}_{p' = \infty, q = \infty}^{s}(\tilde{L})$ for $0 < s < 1$ and $v \geq 2$ regardless of the relation between $d$ and $v$. Thus to answer the question posed previously, by modifying $\bar{P}_n$ so that the number bins in the histogram depends on the regularity $s$, it is possible to achieve minimax rate optimality in the $d > 2v$ case.

So while $\ddot{P}_n$ achieves minimax optimality at least up to logarithmic terms when $d > 2v$ in the H\"{o}lder distribution classes and we expect that this is not so for $\bar{P}_n$, we emphasize that both $\bar{P}_n$ and $\ddot{P}_n$ achieve the minimax rate when $d \leq 2v$ in the H\"{o}lder classes. Furthermore, the alternate estimator $\ddot{P}_n$ is \textit{not} adaptive to the smoothness in that its construction relies on knowledge of $s$, which is usually unavailable in practice; the estimators constructed in this paper, $\bar{P}_n$ and $\Pi_n$, do not suffer from this problem. In these respects $\bar{P}_n$ is superior to $\ddot{P}_n$ in the $d < 2v $ case. Future work will investigate how to maintain adaptivity while obtaining minimax optimality in the $d > 2v$ case for the H\"{o}lder classes.

\subsubsection{Memory efficiency}

Storing $\bar{P}_n$ and $\Pi_n$ only requires maintaining the prior concentration and number of samples associated with each bin. The number of bins is
\[
b_n^d = 2^{\ceil{\log_2(k_n)}d} \lesssim  k_n^d =  \begin{cases} n^{d/2v} & d \leq 2v. \\
n & d > 2v. \end{cases}
\]
Thus, in the case that $d < 2v$, $\bar{P}_n$ provides a polynomial memory improvement over $\hat{P}_n$, and like $\hat{P}_n$, $\bar{P}_n$ also obtains the minimax rate in the class $\mathcal{P}_d$. Moreover, if the sample size $n$ is known prior to collecting data, then a simple binary comparison algorithm (recursive comparison to midpoints) can be used to iteratively place each point $Y_i$ into its appropriate bin. Since there are $2^{K_n}$ bins along each axis and $K_n \lesssim \log(n)$, the total time needed to construct the memory efficient representation is $\lesssim n \log^d(n)$ where the power $d$ comes from performing the binary comparison algorithm along each axis. In particular only a logarithmic time penalty is paid for a polynomial memory reduction.

As discussed in the previous section, whenever $d \leq 2v$, $\bar{P}_n$ and $\Pi_n$ achieve frequentist and posterior contraction rate optimality when considering the $\mathcal{B}_{p',q}^{s}(\tilde{L})$ class for $s >0, 1 \leq p',q \leq \infty$ and \hbox{$\tilde{L}$} sufficiently large. Further, when restricting to H\"{o}lder classes $\mathcal{C}^{s}(\tilde{L})$ with $0 < s < 1$, the histogram estimator $\ddot{P}_n$ of \cite{niles2022minimax} also achieves rate optimality when $v \geq 2$, but does so by using $\asymp n^{d/(d+s)}$ bins. (We refer the reader to Appendix Section \ref{app:nwbConnect}, where we show that this estimator is indeed a dyadic histogram where all prior concentrations are zero and $k_n \asymp n^{1/(d+s)}$). In theory, $\ddot{P}_n$ also provides a polynomial memory improvement over $\hat{P}_n$. As $s \to 0$ there is no improvement, but as $s \to 1$ the factor of reduction in memory footprint relative to $\hat{P}_n$ tends to $n^{1/(d+1)}$. However, in practice, $s$ is usually unknown, rendering this estimator impossible to construct. Moreover, in the $d < 2v$ scenario, $n^{d/2v}$ is smaller than $n^{d/(d+s)}$ by a polynomial for $s <1$, thus even if $s$ is known prior to collecting data, $\ddot{P}_n$ would still require polynomially more memory than $\bar{P}_n$ and $\Pi_n$.

\subsubsection{Prior constraints}
Here, we recall that the upper bound constraint on the prior that appears in Theorems \ref{frequentistDwasser} and \ref{concludingTheoremDwasser} is
\begin{equation}
\label{eqn:priorConstraint}
\sum_{\pmb{j} \in [b_n]^d} \alpha_{\pmb{j},b_n} \lesssim 
\begin{cases}
n^{1/2} & d \leq 2v. \\
n^{1-\frac{v}{d}} & d > 2v.
\end{cases}
\end{equation}
Thus asymptotically in $n$, the specific arrangement of the prior concentrations is unimportant, only the total prior concentration. Moreover the number of prior concentrations is $b_n^d$ where  
\newline \begin{eqnarray}b_n^d \asymp \begin{cases} n^{d/2v} & d \leq 2v. \\ n & d > 2v.\end{cases} \end{eqnarray}
In finite samples, the practitioner may be interested in encoding specificity through the prior. One choice of prior concentrations that may be desired is to set all concentrations equal. If for a given $n$, all prior concentrations are set to a constant $c(n) >0$, the sample size $n$ prior mean histogram is the uniform distribution. For the practitioner with little apriori knowledge about the distribution to be estimated, this is one way of encoding a vague prior. When $d \leq v$, setting $c(n) \equiv c$ is possible because $n^{d/2v} \leq n^{1/2}$. For $v < d \leq 2v$, setting $c(n) = n^{1/2-d/2v}$ is sufficient. For $d > 2v$, similarly setting $c(n) = n^{-v/d}$ will work. We note that even though $c(n) \to 0$ in these latter cases, the shape of the prior is unaffected.

As discussed earlier, a proper prior is not necessary for use of $\bar{P}_n$. Thus if desired, one can set all prior concentrations to zero if analysis will only involve $\bar{P}_n$ but not $\Pi_n$.

\section{Applications to memory efficient inference}
\label{sec:memEff}

In this section, we discuss practical considerations when using $\bar{P}_n$ and $\Pi_n$ for distribution estimation in the batch and streaming data settings in the memory efficient $d < 2v$ case. We also describe the implications of our results for estimation of the Wasserstein distance between distributions and provide instruction on the type of algorithm that should be employed when $\bar{P}_n$ is used for estimating a Wasserstein distance.
\subsection{Batch inference}
In the memory constrained batch inference setting, the only concern of the practicioner is to infer $P_0$ \textit{after} reading all of the data. However, computational memory constraints dictate that at any given time, the amount of data that can be stored is limited. 

If the practitioner knows the sample size before reading the data, then it is clear how to construct the memory efficient histogram representations $\bar{P}_n,\Pi_n$. One can simply use the binary comparison algorithm to drop each data point into one of the $b_n^d$ bins while storing the bin counts on the hardware reading the data. If the sample size is not known, the data can be read twice. On the first pass, the sample size is computed, and on the second pass, the bin counts are computed. Still the runtime is $\lesssim n \log^d(n)$, which is only a logarithmic factor worse than the time needed to construct the empirical measure.

\subsection{Streaming inference}
In the streaming setting, the practitioner needs to infer $P_0$ as the data is being read. In this case, the total size $n$ of the stream is not known at the time of inference and waiting to observe the full stream before conducting inference is not acceptable. To deal with this, the practitioner should provide a conservative upper bound $M$ on $n$. As the stream is being read, the partitions of the space that will be passed through are
\begin{equation}
\mathcal{S}_R = \left\{ A_{\pmb{i},2^{R}}: \pmb{i} \in [2^R] \right\},
\end{equation}
for $R \in \{1,\dots,\ceil{\log_2(k_M)} \}$. To do inference on the fly, suppose that before observing the $r^{th}$ data point $Y_r$ for some $1 \leq r \leq n$, the bin counts for the points $Y_1,\dots,Y_{r-1}$ into each of the partitions $\mathcal{S}_{R}$ for $R \in \{1,2,\dots,\ceil{\log_2(k_M)} \}$ are stored. Then when observing $Y_r$  in the stream, the binary comparison algorithm should be performed to place the $r^{th}$ point into the \textit{finest} partition $\mathcal{S}_{\ceil{\log_2(k_M)}}$. The binary comparison algorithm (recursive comparison to midpoints of search intervals) will ensure that $Y_r$ is also placed into each of the partitions $\mathcal{S}_{R}$ for $R \in \{1,\dots,\ceil{\log_2(k_M)}-1 \}$. Then to perform inference after placing $Y_r$, one can simply compute $R_{r} := \ceil{\log_2(k_r)}$ and use the bin counts from partition $R_{r}$.

In this streaming algorithm, the memory footprint in the $d < 2v$ case, where $k_M = M^{1/2v}$, is $\leq \sum_{R =1}^{\ceil{\log_2(k_M)}} 2^{R d} \leq 2 (2^{\ceil{\log_2(k_M)}d}) \leq 2^{d+1} M^{d/2v}$.
Moreover, the total running time after observing the entire stream is proportional (ignoring constants not depending on $n$ or $M$) to $n \log^d(M)$. As $M$ approaches $n$, these memory and runtime costs are asymptotically (in $n$) the same as those of the batch procedure.

\subsection{Use in estimating Wasserstein distance}
Let $\mathcal{D}_d \subseteq \mathcal{P}_d$ denote the discrete distributions on $[0,1]^d$. For $P_0 \in \mathcal{P}_d$, by the reverse triangle inequality,
\begin{equation}
    \label{estWassDist}
    \sup_{P_1 \in \mathcal{P}_d } \mathbb{E}_{P_0} |W_v(P_0,P_1) - W_v(\bar{P}_n,P_1)| \leq \mathbb{E}_{P_0} W_v(P_0,\bar{P}_n).
\end{equation}
Thus for example, if $P_1 \in \mathcal{P}_d$ is known and $P_0$ is unknown, the convergence rate for estimating $W_v(P_0,P_1)$ is no worse than that of estimating $P_0$ from $\bar{P}_n$ under $W_v$; in particular the convergence rates of Theorem \ref{frequentistDwasser} apply.

 If $P_1 \in \mathcal{D}_d$, the practitioner has two options for computing $W_v(\bar{P}_n,P_1)$. The first is via semi-discrete optimal transport as described in \cite{merigot2011multiscale,kitagawa2019convergence}. Semi-discrete optimal transport methods allow for calculation of the Wasserstein distance between a discrete and continuous distribution; in particular the histogram can be used as the continuous distribution. However, if $\bar{P}_n$ is modified to be a discrete measure where the mass within each bin is concentrated at a particular location within the bin, the convergence rates of Theorem \ref{frequentistDwasser} still hold because at all resolutions considered in the proof, the estimator is unchanged. This modification allows for fully discrete optimal transport computations, in which the reduced memory footprint of $\bar{P}_n$ may yield run time benefits. 
In simulations 3 and 4 of Section \ref{sec:sims}, we experimentally compare semi-discrete optimal transport to discretization of the histogram followed by discrete-discrete optimal transport, and these experiments suggest that at small sample sizes, the former method statistically outperforms the latter for the distribution estimation problem. 

As discussed in more detail in the following section, the correctness of experimentally efficient semi-discrete optimal transport algorithms rely on an assumption about the connectedness of the support of the density \citep{kitagawa2019convergence}. Therefore, when estimating distributions with non-connected support with the intent to compare to known discrete distributions, setting the prior concentrations to be non-zero in $\bar{P}_n$ is pivotal in ensuring Wasserstein computation using semi-discrete optimal transport is possible.

Another scenario of interest is when both $P_0,P_1 \in \mathcal{P}_d$ are unknown and the task is to estimate $W_v(P_0,P_1)$. In this case, if $n$ independent samples are collected from $P_0$ and $P_1$ respectively and $P_0 \times P_1$ is the product measure, then again by the reverse triangle inequality
\begin{equation}
    \mathbb{E}_{P_0 \times P_1}|W_v(P_0,P_1) - W_v(\bar{P}_{0n},\bar{P}_{1n})| \leq \mathbb{E}_{P_0}W_v(\bar{P}_{0n},P_0)+\mathbb{E}_{P_1} W_v(\bar{P}_{1n},P_1).
\end{equation}
Thus the convergence rate for estimating $W_v(P_0,P_1)$ is not worse than the slower of the rates for estimating the distributions of $P_0$ and $P_1$. In particular the convergence rates of Theorem \ref{frequentistDwasser} apply. In computation, the practitioner can discretize $\bar{P}_{0n}$ and $\bar{P}_{1n}$ by placing the mass of each bin at a single point within the bin, and then use a standard discrete-discrete optimal transport algorithm.
\section{Simulations}
\label{sec:sims}
The  theory we have presented accompanied by the existing minimax lower bounds of \cite{singh2018minimax} and \cite{niles2022minimax}  identify classes of distributions in which the estimator $\bar{P}_n$ is minimax rate optimal. In this section, we use simulations to compare $\bar{P}_n$ to existing frequentist minimax procedures for distribution estimation under $W_v$.

In the following, we consider only the case $d < 2v$, where the number of bins in $\bar{P}_n$ is $2^{\lceil \log_2( n^{d/2v} ) \rceil} \lesssim n^{d/2v}.$ This is the memory efficiency case where the storage requirement of the histogram yields a polynomial $n^{1-d/2v}$ memory improvement over the empirical measure, and a polynomial memory improvement over the histogram presented in \cite{niles2022minimax}. The comparisons we show are between \hbox{$\bar{P}_n$}, at various different prior concentration settings, and the empirical measure. While the empirical measure is chosen because it is a minimax rate optimal procedure \hbox{\citep{singh2018minimax}}, we do not consider the histogram presented in \cite{niles2022minimax} since it is not adaptive to the regularity $s$. This is an important consideration, since there is frequently limited knowledge that a general distribution $P_0$ admits an $s \in (0,1]$ H\"{o}lder regular density, rendering construction of this estimator infeasible.

We hypothesize that there are at least two types of data generating distributions in which $\bar{P}_n$ will perform just as well, or better than the empirical measure. The first is when $P_0$ is close to the uniform distribution, which is a special instance of the posterior mean histogram at any sample size $n$, when all posterior weights are equal.
The second case is when distributions have areas of zero mass in between areas of positive mass. For one dimensional distributions $P_0$ with densities $p_0$, \citet[Chapter~5]{bobkov2019one} define the quantity 
$$J_v(P_0) := \int_{0}^{1} \frac{[F_0(x)(1-F_0(x))]^{v/2}}{p_0(x)^{v-1}} \mr{d}x, $$
and show that the empirical measure will achieve the standard rate $n^{-1/2}$ provided $J_v(P_0) < \infty$. Furthermore, \citet[Chapter~7]{bobkov2019one} prove that for any probability distribution $P_0$ on $\mathbb{R}$ whose support is not an interval, the rate achieved by the empirical measure is $n^{-1/2v}$. In one dimension, since the upper bound $n^{-1/2v}$ for the memory efficient histogram has been established, we look to distributions with disconnected support for instances in which the memory efficient histogram and empirical measure will certainly perform similarly (or where the memory efficient histogram will perform better). For the two dimensional case, there is also evidence that the minimax rate can decay faster than $n^{-1/2v}$ when considering classes of distributions with some strong characterization of connectedness in their support. For example, \citet[~Theorem 2]{niles2022minimax} give an estimator that achieves the $n^{-1/2}$ rate up to logarithmic terms for classes of smooth densities that are bounded below by a positive constant. Construction of this estimator requires knowledge of the $L^v$ norm of the density to be estimated; due to this lack of adaptivity and likely unavailability of $L^v$ norm information in the density estimation setting, we do not consider this estimator in our simulations. However, the existence of this estimator suggests that even in higher dimensions, within the class $\mathcal{P}_d$, we should investigate distributions that either have a density touching zero or that have disconnected support for instances in which the memory efficient histogram will  perform similarly or better than existing estimators in this problem space. It may be true that like the empirical measure in one dimension, the convergence rate of $\mathbb{E}_{P_0}W_v(P_0,\bar{P}_n)$ can quicken when $P_0$ has a density that is bounded below, but we leave this investigation for future work.

\subsection{One dimensional simulations}
For one dimensional simulations, in order to compute the Wasserstein distance, we utilize the quantile function based representation of $W_v$ for two probability measures $P_1$ and $P_2$ on $[0,1]$ \citep[Chapter 2]{bobkov2019one}. Specifically, letting $F_1^{-1}(z) := \inf \{x \in [0,1]: F_1(x) \geq z\}$ and $F_2^{-1}(z) = \inf\{x \in [0,1]: F_2(x) \geq z\}$, we have that $W_v^v(P_1,P_2) = \int_{0}^{1} |F_1^{-1}(z) - F_2^{-1}(z)|^v \mr{d}z$. This representation of $W_v$ permits exact or nearly exact computation when comparing two distributions with easily computable quantile functions.  Additionally, the quantile function of $\bar{P}_n$ is piecewise linear between the end points of the bins. Letting $\bar{F}_n^{-1}$ denote the quantile function of $\bar{P}_n$, we have that
\begin{equation}
    \bar{F}_n^{-1}(z) = \sum_{j=1}^{b_n} \mathbb{I}\left(\frac{j-1}{b_n} \leq z < \frac{j}{b_n} \right) \left[\frac{j-1}{b_n}+\frac{z-\sum_{t=1}^{j-1} \hat{\pi}_{t,b_n}}{b_n \hat{\pi}_{j,b_n}}\right],
\end{equation}
where for $1 \leq j \leq b_n,\hat{\pi}_{j,b_n}$ represent the bin probabilities computed for the histogram $\bar{P}_n$.

In the following, we provide two different one dimensional simulations. In both, the \verb|integrate| function found in the \verb|stats| package in R is used to numerically approximate the integral. In some instances, we consider distributions such as the uniform that admit a simple enough quantile function where analytic computation is possible. However, to be consistent across all examples, numerical integration is always used. In all simulations, we consider sample sizes $\log_2(n) \in \{2,4,6,8,10,12\}$, and for each distribution $100$ Monte Carlo samples are collected to estimate $\log_2( \mathbb{E}_{P_0} W_v(\bar{P}_n,P_0))$ and $\log_2(\mathbb{E}_{P_0} W_v(\hat{P}_n,P_0))$. For $\bar{P}_n$, two different priors are considered. The first is where all prior concentrations are zero; the second is when all prior concentrations are $n^{0.5-1/2v}$, thereby ensuring that the total prior concentration satisfies $\lesssim n^{-1/2}$ as is required by Theorem \ref{frequentistDwasser}. Finally, note that in all one dimensional simulations, estimates of the aforementioned quantities are supplemented with Delta method based $95\%$ confidence intervals.

\subsubsection*{Simulation 1 (exploring deviations from uniformity)}
The first simulation explores near uniform distributions using the Beta kernel. The data generating distributions considered are $P_0 \in \Beta(x,x)$ for $x \in \{0.7,0.9,1,1.1,1.3\}$ and $v \in \{1,2,3\}$. In these cases it is straightforward to show that $J_v(P_0) < \infty$. Therefore, the empirical measure achieves the standard convergence rate $n^{-1/2}$. While the exact convergence rate of $\bar{P}_n$ is unknown in these instances, we still expect to see similar performance since distributions ``close'' to the uniform should favor the histogram model by construction. Figure \ref{fig:1dBeta} displays comparisons between the empirical measure, its upper bound (``worst case''), the memory efficient histogram with zero prior concentration and the memory efficient histogram with prior concentrations equivalent to $n^{0.5-1/2v}$. This ``worst case'' upper bound for $\mathbb{E}_{P_0} W_v(\hat{P}_n,P_0)$ in one dimension follows from \citet[Theorem 3.2]{bobkov2019one} and an application of Jensen's inequality.
It is observed that $\bar{P}_n$ performs at least as well as the empirical measure at all sample sizes when $x \in \{0.9,1,1.1\}$. However, when the deviation from the the uniform distribution is larger, there are instances, for example with $v=3,x = 0.7, 1.3$, when the empirical measure outperforms $\bar{P}_n$ for sample sizes larger than $n = 2^8$.
\begin{figure}
    \centering
    \includegraphics[height=1.05\textheight]
{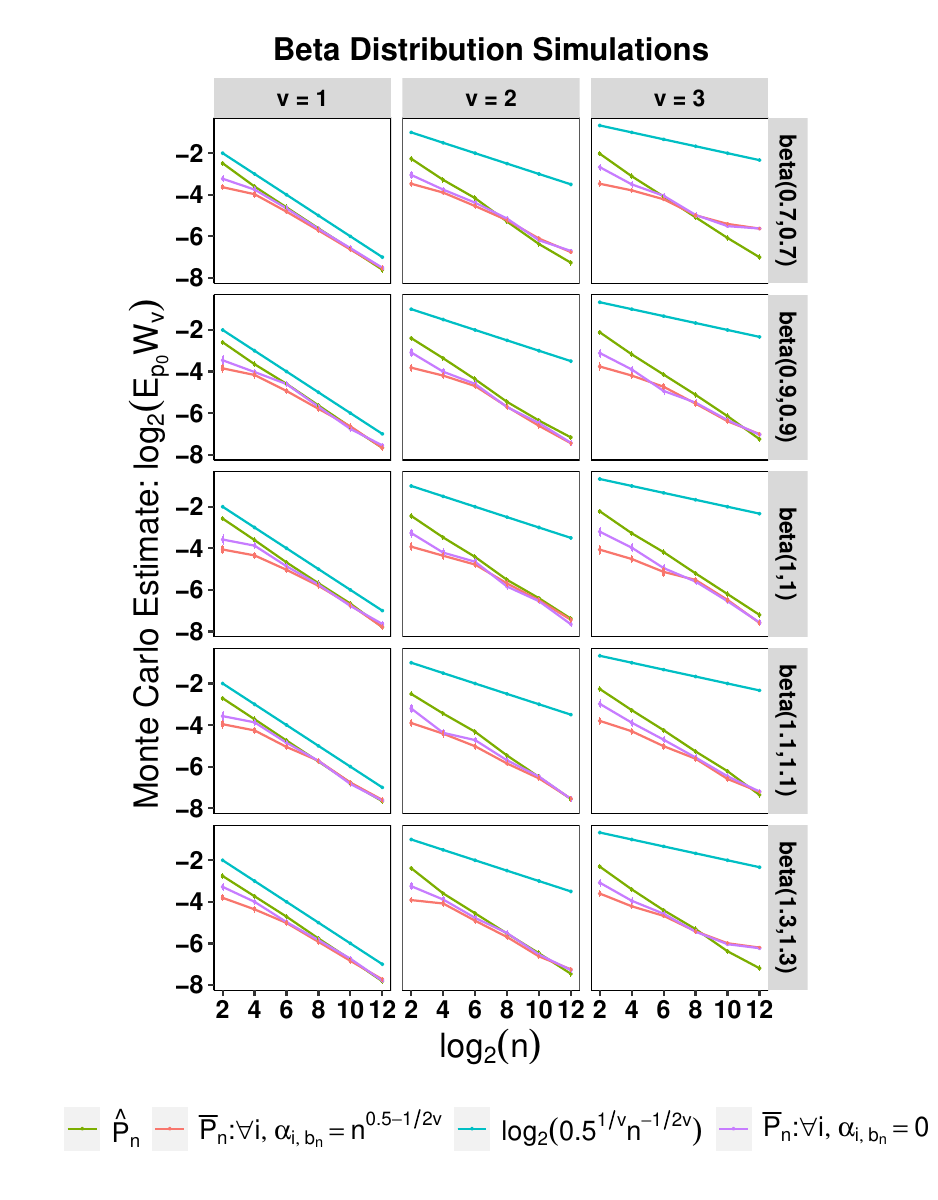}
    \caption{100 Monte Carlo Sample Estimate of $\log_2(\mathbb{E}_{P_0} W_v(\cdot,P_0))$ for the empirical measure $\hat{P}_n$ and for the memory efficient histogram $\bar{P}_n$ where $P_0 = \beta(x,x)$ for $x \in \{0.7,0.9,1,1.1,1.3\}$. $95\%$ confidence intervals computed using the Delta method are displayed. An upper bound for the worst case behavior of the empirical measure, $\log_2(0.5^{1/v} n^{-1/2v})$, is also plotted.}
    \label{fig:1dBeta}
\end{figure}

\subsubsection*{Simulation 2 (Exploring a distribution with disconnected support)}
In the second simulation study, we explore distributions that have a positive Lebesgue measure area of zero mass. To do so, we consider the class of densities $\mathcal{S} = \{p_{a,e}(x)| 0 < e < .5, 0 < a < 1/e^2, b = (1-ae^2)/2e \}$ where $p_{a,e}(x) \in \mathcal{S}$ follows
\begin{equation*}
p_{a,e}(x) = \mathbb{I}(0 \leq x < e) (ax+b)+\mathbb{I}(1-e \leq x < 1)(ax+b-(1-e)a).
\end{equation*}
From herein, we denote the $\split(a,e)$ distribution as that which admits density $p_{a,e}(x) \in \mathcal{S}$.

The \hbox{$e$} parameter controls the gap between positive and zero measure areas permitting exploration into whether increasing the size of the zero mass area changes performance of the methods under comparison. In particular, this gap increases as $e \to 0$. The $a$ parameter gives the slope of the density when it is non-zero, allowing exploration of how robust the histogram is to deviations from uniformity. Setting $a$ as large as possible gives a highly non-uniform density, while $a$ near zero gives uniformity in the regions possessing positive mass. Examples of different densities admitted by the Split distribution are shown in Figure \ref{fig:splitDistrExamples}. The second simulation explores $P_0 \in \split(a,e)$ for $(a,e) \in \{(a_i,e_j)\}$ for $i \in \{1,2,3\}$ and $j \in \{1,2\}$ where $e_1 = 0.1, e_2 = 0.27, a_1 = 0.5, a_2 = 2, a_3 = 13$. These results are displayed in Figure \ref{fig:splitDistrResults} and show comparisons with the same quantities as presented in Figure \ref{fig:1dBeta}. It is observed that across the board, the performance of $\bar{P}_n$ follows closely with the empirical measure and in many instances as can be seen in the $v=3$ setting, shows a consistent improvement across all studied $a$ values. Even when areas of positive mass have a highly nonuniform density, for example when $a = 13$, the memory efficient histogram still performs similarly and sometimes even better than the empirical measure.

In both simulations one and two, also observe that the width of the Delta method based $95\%$ confidence intervals are minuscule relative to the magnitude of decay in the error observed as the sample size increases.

\begin{figure}[h!]
    \includegraphics[scale=1.03]{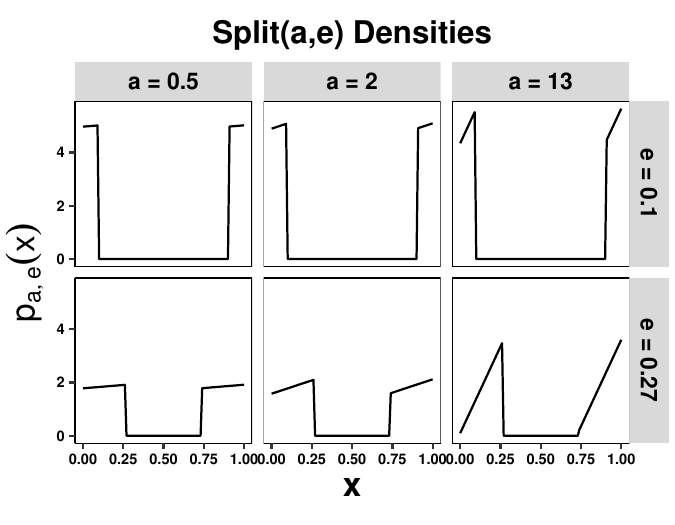}
    \centering
    \caption{ Densities from the $\split(a,e)$ distribution for the choices of $a,e$ used in simulations in Figure \ref{fig:splitDistrResults}.}
    \label{fig:splitDistrExamples}
\end{figure}

    

\begin{figure}[p]
\vspace*{-4cm}
\includegraphics[height=1.25\textheight]{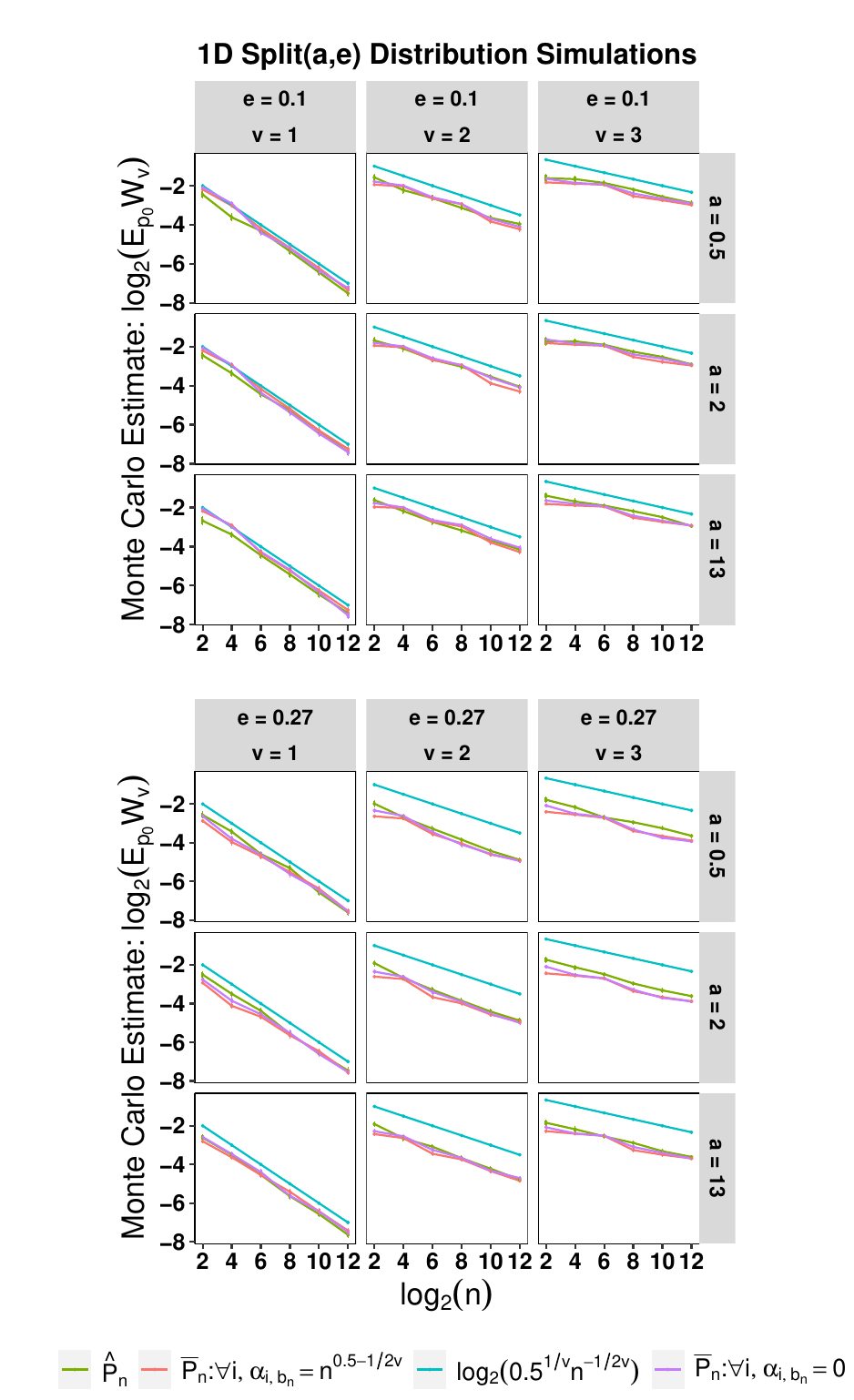}
    \centering
    \caption{100 Monte Carlo sample simulation of $\log_2(\mathbb{E}_{P_0} W_v(\cdot,P_0))$ for the empirical measure $\hat{P}_n$ and for the memory efficient histogram $\bar{P}_n$ where $P_0 =  \split(a,e)$ for various values of $a,e$. Figure \ref{fig:splitDistrExamples} provides visuals of these distributions. $95 \%$ confidence intervals based on the Delta method are displayed. An upper bound for the worst case behavior of the empirical measure, $\log_2(0.5^{1/v} n^{-1/2v})$, is also plotted.}
    \label{fig:splitDistrResults}
\end{figure}

\subsection{Two dimensional simulations}
In the two dimensional case, the convenient quantile function based integral expression for the Wasserstein distance does not exist. Computation of the Wasserstein--2 distance between a probability measure possessing a density and a discrete distribution is studied in the semi-discrete optimal transport work of \cite{merigot2011multiscale}. A more  experimentally efficient algorithm is presented in \cite{kitagawa2019convergence}, and an implementation in line with these works is provided in the \verb|pysdot| package in Python. 



\sloppy The sample sizes $\log_2(n) \in \{4,6,8,10,12\}$ are considered, and $100$ Monte Carlo samples taken directly from the continuous measure are used to estimate $\log_2(\mathbb{E}_{P_0} W_2(\bar{P}_n,P_0))$ and $\log_2(\mathbb{E}_{P_0} W_2(\hat{P}_n,P_0))$. 
To achieve the semi-discrete optimal transport setting, we use a 1000 sample empirical construction of \hbox{$P_0$} as ground truth for Wasserstein computations. Due to the simulation setup, therefore, the data generating distribution and the distribution utilized for Wasserstein computations are not identical. However, given the sample sizes under study, this does not prevent us from performing a meaningful analyses.

In addition to $\bar{P}_n$, we also study the behavior of the discretization of $\bar{P}_n$ obtained by placing all of the mass of each bin at its center. We call this estimator $\tilde{P}_n$. Theorem \ref{frequentistDwasser} also holds for $\tilde{P}_n$ and the discreteness of $\tilde{P}_n$ provides the practitioner with additional options if they intend to use the memory efficient histogram in computation. For comparison of two discrete measures, which occurs in simulations comparing $\hat{P}_n$ to $\tilde{P}_n$, we use the \verb|Transport| package in R \citep{schuhmacher2023package} which provides an implementation of discrete-discrete optimal transport in the \verb|wasserstein| function.

\subsubsection*{Simulations 3 and 4 (exploring the product uniform and product split distributions)}
In the following simulations, we consider the two dimensional uniform distribution and the product measure $\split(a = 2,e=0.27) \times \split(a=2,e=0.27)$ in the $v=2$ case. Results and comparisons are displayed in Figure \ref{fig:2dDistrResults}. It is immediately obvious that in both cases the non-discretized memory efficient histogram $\bar{P}_n$ outperforms the empirical measure $\hat{P}_n$. Moreover, we observe a penalty for discretization of $\bar{P}_n$ since $\hat{P}_n$ outperforms $\tilde{P}_n$. This penalty is more pronounced when estimating the smooth uniform distribution than it is when estimating the product $\split(a=2,e=0.27)$ distribution. 

For the product $\split(a=2,e=0.27)$ distribution, results for the zero prior memory efficient histogram are not displayed since semi-discrete optimal transport output from \verb|pysdot| is not reliable. This is because, as is well documented in \cite{kitagawa2019convergence}, convergence guarantees of the gradient based optimization routines used in semi-discrete optimal transport rely on an assumption about the connectedness of the support of the density. Due to these computational challenges, the prior serves a very practical purpose even in a frequentist analysis. By setting all prior concentrations positive, the histogram has convex support, resolving any issues with convergence that could arise due to a disconnected support. The interested reader is directed to \cite{peyre2019computational,kitagawa2019convergence} for a further discussion of semi-discrete optimal transport algorithms.

Finally, note that as in simulations one and two, in simulations three and four, the width of the $95\%$ confidence intervals produced using the Delta method are minuscule relative to the magnitude of decay in the error observed as the sample size increases.

\begin{figure}[h!]
    \includegraphics[scale=.5]{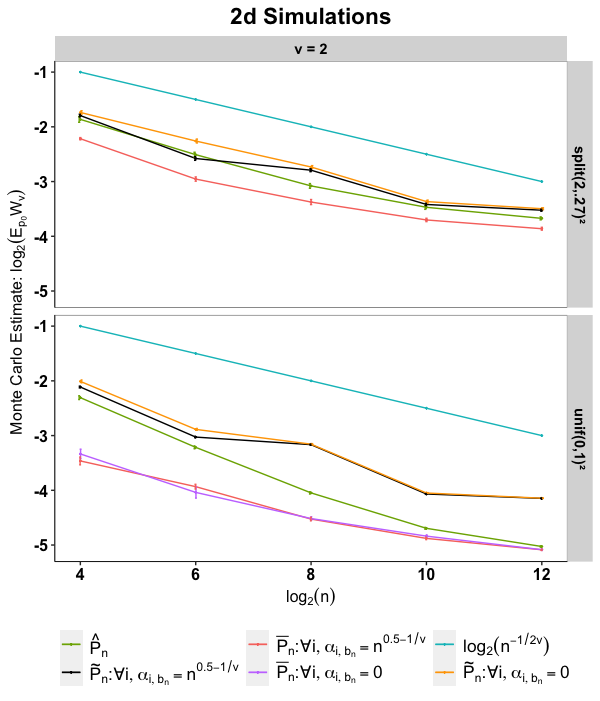}
    \centering
    \caption{100 Monte Carlo sample simulation of $\log_2(E_{P_0} W_2(\cdot,P_0))$ for the empirical measure $\hat{P}_n$ and for the memory efficient histogram $\bar{P}_n$ where $P_0$ is the 2 dimensional uniform or the product measure $\split(a=2,e=0.27) \times \split(a=2,e=0.27)$. $95 \%$ confidence intervals based on the Delta method are displayed. While Monte Carlo samples involve taking $n$ samples directly from the data generating measure, for Wasserstein computation, the data generating measures have been discretized using $1000$ random samples and this accounts for the  curves leveling off for large $n$. To give context for the rates observed, the $\log_2(n^{-1/2v})$ curve is also plotted.}
    \label{fig:2dDistrResults}
\end{figure}

\section{Conclusion}
\label{sec:conclusion}

On the metric space $([0,1]^d,\| \cdot - \cdot \|_p)$, we have proven upper bounds on the rate of posterior contraction for the posterior of the Bayes histogram and on the rate of the expected loss for an estimator derived from the Bayes histogram ($\bar{P}_n$), all under the Wasserstein distance ($W_v,1 \leq v < \infty)$.

These rates match existing minimax lower bounds at least up to logarithmic terms in the class $\mathcal{P}_d = \{ \text{Borel Probability Measures on } [0,1]^d \}$ and when $d < 2v$, in the Besov $\mathcal{B}_{p',q}^{s}(\tilde{L})$ space where $1 \leq p',q \leq \infty, s >0$ and $\tilde{L}$ sufficiently large. Our posterior contraction proofs take advantage of conjugacy instead of using the well-known three condition \cite{ghosal2000} approach. We therefore sidestep common challenges that may arise when dealing with metrics that either have a faster minimax convergence rate than under Kullback-Leibler, or are not dominated by the Hellinger metric.

Our results are of most practical interest in the $d < 2v$ case, where $\lesssim n^{d/2v}$ atoms are required to store $\bar{P}_n$ and $\Pi_n$. Compared to the empirical measure this is an $n^{1-d/2v}$ polynomial factor improvement. In the $d < 2v$ case, our frequentist histogram $\bar{P}_n$ is superior to the histogram $\ddot{P}_n$ put forth by \cite{niles2022minimax}, since unlike $\ddot{P}_n$, the construction of $\bar{P}_n$ is not dependent on knowledge of a regularity parameter $s$ of $P_0$, yet still obtains minimax optimality in the same regularity classes as $\ddot{P}_n$ at least up to logarithmic terms. Further, the memory footprint of $\bar{P}_n$ is always at least as small as that of $\ddot{P}_n$.

The $d < 2v$ case is practically important due to the popularity of the $W_1$ and $W_2$ metrics. This case covers $(d =1, v \geq 1)$, $(d=2,v=2)$ and $(d=3,v=2)$. In particular, by covering the $d=2$ and $d=3$ cases, it is possible to use $\bar{P}_n$ and $\Pi_n$ for image and video analysis respectively.

Potential applications of the memory efficient Bayes histogram include computationally memory constrained inference settings in which storing the entire data set in memory is not possible, as well as any of the myriad of problems (such as MWDE and Wasserstein ABC) where one needs a low atom representation of $P_0$ in order to reduce Wasserstein computation time. As observed in Sections \ref{sec:memEff} and \ref{sec:sims}, the prior plays an important role in facilitating Wasserstein computations using semi-discrete optimal transport.

An important avenue for future work will be to compare the rate of convergence in estimating a Wasserstein distance via the plug--in approach using the memory efficient histogram to other existing procedures for estimating a Wasserstein distance. One theoretical limitation of our work is that there are certain classes of distributions, such as those possessing densities that are bounded below or in the one-dimensional case those such that the constant $J_v(P_0) < \infty$, where estimators have been identified that achieve exactly or nearly the standard rate $n^{-1/2}$ for the distribution estimation problem. Another direction for future work, therefore, is to investigate whether the memory efficient histogram can also benefit from such an assumption. A practical limitation of our work is that in the streaming setting, one needs a conservative upper bound on the total stream size to perform inference along the stream. To deal with this in subsequent work, we may attempt to place a prior on the number of bins such that the memory efficiency property is maintained while the requirement to have some knowledge of the sample size $n$ before constructing the histogram, is removed.

\section*{Acknowledgements}
We would like to thank Bei Wang at the University of Utah for providing funding support under contract DOE DE-SC0021015.
This paper describes objective technical results and analysis. Any subjective views or opinions that might be expressed in the paper do not necessarily represent the views of the U.S. Department of Energy or the United States Government. This work was supported by the Laboratory Directed Research and Development program at Sandia National Laboratories, a multimission laboratory managed and operated by National Technology and Engineering Solutions of Sandia, LLC, a wholly-owned subsidiary of Honeywell International, Inc., for both the U.S. Department of Energy’s National Nuclear Security Administration under contract DE-NA0003525.

\clearpage 
\newpage
\appendix
\begin{center}
{\bf\Large Appendix}
\end{center}

\section{Proof of Result \ref{diricheletConcentration}}
\label{app:dConcentration}

\diricheletConcentration*
\begin{proof}
Basic properties of the Dirichlet distribution give that for $j \in \{1,2,\dots,k\}$, $\pi_j \sim \Beta(\alpha_j,\bar{\alpha}-\alpha_j)$. Also, if $X \sim \Beta(\alpha,\beta)$ then $\Var(X) = \alpha \beta/((\alpha+\beta)^2 (\alpha+\beta+1))$. Using these properties, in addition to Jensen's inequality and Cauchy--Schwarz inequality, we have that
\begin{align}
\mathbb{E} \left( \sum_{j=1}^{k} |\pi_j - \mathbb{E}(\pi_j) | \right) & \leq \sum_{j=1}^{k} \sqrt{\Var(\pi_j)} \nonumber \\
& = \sum_{j=1}^{k} \sqrt{ \frac{\alpha_j(\bar{\alpha} - \alpha_j)}{\bar{\alpha}^2(\bar{\alpha}+1)} } \nonumber \\
& \leq (\bar{\alpha})^{-\frac{3}{2}} \sum_{j=1}^{k} \sqrt{\alpha_j(\bar{\alpha}-\alpha_j)} \nonumber \\ 
& \leq (\bar{\alpha})^{-\frac{3}{2}} \sqrt{\left(\sum_{j=1}^{k} \alpha_j \right) \left(\sum_{j=1}^{k} \bar{\alpha} - \alpha_j \right)} \nonumber \\
& = (\bar{\alpha})^{-\frac{3}{2}} \sqrt{\bar{\alpha} (\bar{\alpha} k - \bar{\alpha})} \nonumber \\
& \leq (\bar{\alpha})^{-\frac{1}{2}}  \sqrt{k}.
\end{align}
By Markov the result follows.
\end{proof}

\section{Proofs of Theorems \ref{frequentistDwasser}, \ref{contractionAroundCentralEstimatorDwasser}, and \ref{concludingTheoremDwasser}}
\label{app:mainTheoremsProofs}

\subsubsection{Theorem \ref{frequentistDwasser}}

To apply the multiresolution upper bound given in Result \ref{wassersteinMultiRes}, we will need to pick a sequence of partitions and choose a depth to analyze.  By setting up the model space to be a sequence of nested dyadic histograms, it is natural to use the sequence of dyadic partitions of $[0,1)^d$. The following Lemma formally puts forth the partition sequence we will be using throughout the proofs. Part 1 of Lemma \ref{nestedPartitionDetails} establishes (the intuitively obvious fact) that the sequence of partitions we will use is nested, and the resolution of each partition is computed. Part 3 of Lemma \ref{nestedPartitionDetails} ensures that the partition at the depth of the model satisfies that sets in courser partitions are always just unions of sets from this finest partition. This will be useful in the proof of Theorem \ref{frequentistDwasser}. Part 2 of Lemma \ref{nestedPartitionDetails} is an intermediate result useful in proving part 3.

\begin{lemma}
\label{nestedPartitionDetails}
With $\mathcal{S}_0 := [0,1)^d$ and 
\[\mathcal{S}_k := \left\{ \left[\frac{i_1-1}{2^k},\frac{i_1}{2^k}\right) \times \left[\frac{i_2-1}{2^k},\frac{i_2}{2^k}\right) \times \dots \times \left[\frac{i_d-1}{2^k},\frac{i_d}{2^k}\right) \textit{ for } (i_1,i_2,\dots,i_d) \in [2^k]^d\right\}\] for $k \in \mathbb{N}$, the following holds
\begin{enumerate}
 \item for $k \in \{0,1,2,\dots,\}$, under the $1 \leq p < \infty$ norm $\| \cdot \|_p$, $Res(\mathcal{S}_k) = d^{1/p} 2^{-k}$. Also, for $k \in \{1,2,\dots,\}$, if $S' \in \mathcal{S}_k$, there is a $S \in \mathcal{S}_{k-1}$ such that $S' \subseteq S$ [establishing that for each $K \in \{1,2,\dots\}$, $\{ S_k \}_{k=0}^{K}$ is a \textit{sequence of nested partitions}]. Also, for any $S^* \in \mathcal{S}_{k-1}$, s.t $S^* \neq S, S' \cap S^* = \emptyset$.
 \item For each $n \in \mathbb{N}$ and each $k \in \{1,2,\dots,K_n\}$ if $S' \in \mathcal{S}_{K_n}$, then there is a $S \in \mathcal{S}_k$ such that $S' \subseteq S$ and for every $S^* \neq S$ such that $S^* \in \mathcal{S}_k$, $S' \cap S^* = \emptyset$
 \item For each $n \in \mathbb{N}$ and each $k \in \{1,2,\dots,K_n\}$ and each $S \in \mathcal{S}_k$, there exists a set $I_{S,k,n} \subseteq [b_n]^d$ and $S = \bigcup_{\pmb{j} \in I_{S,k,n}} A_{\pmb{j},b_n}$. Consequently $\{ \bigcup_{\pmb{j} \in I_{S,k,n}} A_{\pmb{j},b_n} \}_{S \in \mathcal{S}_k}$ partitions $[0,1)^d$ and $\{I_{S,k,n}\}_{S \in \mathcal{S}_k}$ partitions $[b_n]^d$.
\end{enumerate}
\end{lemma}
\begin{proof}
For (1), let $k \in \{0,1,2,\dots,\}$, and suppose $(i_1,i_2,\dots,i_d) \in [2^k]^d$ and $S' = \prod_{v =1}^{d}  [\frac{i_v-1}{2^k},\frac{i_v}{2^k})$. Then $diam(S') = (d 2^{-pk})^{1/p} = d^{1/p} 2^{-k}$. This establishes that $Res(\mathcal{S}_k) = d^{1/p} 2^{-k}$. Now for $k \in \{1,2,\dots\}$, define $f_k: \{1,2,\dots,2^k\} \to \{1,2,\dots,2^{k-1} \}$, by
\[
f_k(i) := \sum_{j = 1}^{2^{k-1}} j *\mathbb{I} \left(\frac{i}{2^k} \in [\frac{j-1}{2^{k-1}},\frac{j}{2^{k-1}} ) \right).
\]
Then for $v \in \{1,2,\dots,d\}$, $\left[\frac{i_v-1}{2^k},\frac{i_v}{2^k}\right) \subseteq \left[\frac{f_k(i_v)-1}{2^{k-1}},\frac{f_k(i_v)}{2^{k-1}} \right)$. In particular, 
\[
S' = \prod_{v=1}^{d} \left[\frac{i_v-1}{2^k},\frac{i_v}{2^k}\right) \subseteq \prod_{v=1}^{d} \left[\frac{f_k(i_v)-1}{2^{k-1}},\frac{f_k(i_v)}{2^{k-1}}\right) := S \in \mathcal{S}_{k-1}.
\]
Now consider $S^* \neq S$, $S^* \in \mathcal{S}_{k-1}$. Since the members of $\mathcal{S}_{k-1}$ are disjoint, we have that $S \cap S^* = \emptyset$ and $S' \subseteq S$. Therefore $S' \cap S^* = \emptyset$. This completes (1). 

For (2), let $n \in \mathbb{N}$; we will do a proof by induction on the values $k \in \{1,2,\dots,K_n\}$. The base case $k = K_n$ is clear since $\mathcal{S}_{K_n}$ is a partition. Now suppose for some $j \in \{0,1,2,K_n-2\}$, we have that for every $S' \in \mathcal{S}_{K_n}$, there is a $S \in \mathcal{S}_{K_n-j}$ such that $S' \subseteq S$ and for every $S^* \in \mathcal{S}_{K_n-j}$ such that $S^* \neq S$, $S^* \cap S' = \emptyset$. Applying (1) with $k = K_n-j-1$ yields that for some $U \in \mathcal{S}_{K_n-j-1}$, $S' \subseteq U$. Since $\mathcal{S}_{K_n-j-1}$ is a partition and in particular consists of disjoint sets, we also have that for all $U^* \in \mathcal{S}_{K_n-j-1}$ such that $U^* \neq U$, $S' \cap U^* = \emptyset$. This argument applies for all $S' \in \mathcal{S}_{K_n}$ and this completes the inductive step. So by induction we conclude (2). 

For (3), let $n \in \mathbb{N}$; recall that by definition of the $A_{\pmb{j},b_n}$ for $\pmb{j} \in [b_n]^d$, we have that $\mathcal{S}_{K_n} = \{ A_{\pmb{j},b_n} \}_{\pmb{j} \in [b_n]^d}$. So let $k \in \{1,2,\dots,K_n\}$ and suppose $S \in \mathcal{S}_{k}$. Further, let $I_{S,k,n} := \{ \pmb{j} \in [b_n]^d: A_{\pmb{j},b_n} \subseteq S \}$. It is clear that $\bigcup_{\pmb{j} \in I_{S,k,n}} A_{\pmb{j},b_n} \subseteq S$. Now for sake of contradiction suppose $S$ is not contained in $\bigcup_{\pmb{j} \in I_{S,k,n}} A_{\pmb{j},b_n}$. Then since $\mathcal{S}_{K_n}$ partitions $[0,1)^d$ and $S \subseteq [0,1)^d$, there must exists a $\pmb{j^*} \in [b_n]^d$, $\pmb{j^*} \notin I_{S,k,n}$ such that $A_{\pmb{j^*},b_n} \cap S \neq \emptyset$. Since $A_{\pmb{j^*},b_n} \nsubseteq S$, by (2), there exists a $Q \in \mathcal{S}_k$ such that $Q \neq S$ and $A_{\pmb{j^*},b_n} \subseteq Q$. Therefore, $S \cap Q \neq \emptyset$. This is a contradiction since $\mathcal{S}_k$ is a partition and $Q,S \in \mathcal{S}_k$. Thus we conclude $S \subseteq \bigcup_{\pmb{j} \in I_{S,k,n}} A_{\pmb{j},b_n}$ and in particular $S = \bigcup_{\pmb{j} \in I_{S,k,n}} A_{\pmb{j},b_n}$. 

As $\mathcal{S}_k$ is a partition of $[0,1)^d$, and for each $S \in \mathcal{S}_k$, $S = \bigcup_{\pmb{j} \in I_{S,k,n}} A_{\pmb{j},b_n}$, we conclude that $\{ \bigcup_{\pmb{j} \in I_{S,k,n}} A_{\pmb{j},b_n} \}_{S \in \mathcal{S}_k}$ partitions $[0,1)^d$. 

Finally, to show that $\{I_{S,k,n}\}_{S \in \mathcal{S}_k}$ partitions $[b_n]^d$, consider $S_1 \neq S_2$, $S_1,S_2 \in \mathcal{S}_k$. Since $\mathcal{S}_k$ is a partition, $S_1 \cap S_2 = \emptyset$. By definition of $I_{S_1,k,n}$, we have that for any $\pmb{j} \in I_{S_1,k,n}$, $A_{\pmb{j},b_n} \subset S_1$. Therefore $A_{\pmb{j},b_n} \cap S_2 = \emptyset$ and in particular $\pmb{j} \notin I_{S_2,k,n}$. Thus $I_{S_1,k,n} \cap I_{S_2,k,n} = \emptyset$. Thus $\{I_{S,k,n}\}_{S \in \mathcal{S}_k}$ is a disjoint collection of sets. Also for $\pmb{j} \in [b_n]^d$, since $A_{\pmb{j},b_n} \in \mathcal{S}_{K_n}$, applying (2) with $S' = A_{\pmb{j},b_n}$ implies that there is an $S \in \mathcal{S}_k$ such that $A_{\pmb{j},b_n} \subseteq S$. Thus $\pmb{j} \in I_{S,k,n}$. Therefore $[b_n]^d \subset \bigcup_{S \in \mathcal{S}_k} I_{S,k,n}$ and in particular $\bigcup_{S \in \mathcal{S}_k} I_{S,k,n} = [b_n]^d$. Thus we conclude $\{I_{S,k,n}\}_{S \in \mathcal{S}_k}$ partitions $[b_n]^d$.
\end{proof}

By Lemma \ref{nestedPartitionDetails} part 1 and Result \ref{wassersteinMultiRes}, we have that on the metric space $([0,1)^d,\| \cdot \|_{p})$ with $1 \leq p < \infty$ where $\mathcal{S}_0 := [0,1)^d$ and for $k \in \mathbb{N}$
\[
\mathcal{S}_k :=  \left\{ \left[\frac{i_1-1}{2^k},\frac{i_1}{2^k}\right) \times \left[\frac{i_2-1}{2^k},\frac{i_2}{2^k}\right) \times \dots \times \left[\frac{i_d-1}{2^k},\frac{i_d}{2^k}\right) \textit{ for } (i_1,i_2,\dots,i_d) \in [2^k]^d \right\},
\]
\begin{equation}
\label{multires}
W_v^v(\mu,\nu) \leq d^{v/p} \left( \left(\frac{1}{2}\right)^{K v} + \sum_{k=1}^{K} \left(\frac{1}{2}\right)^{(k-1) v} \sum_{S \in \mathcal{S}_k} |\mu(S) - \nu(S)| \right).
\end{equation}
for each $K \geq 1$ (and $1 \leq v < \infty$).
We are now ready to prove Theorem \ref{frequentistDwasser}, where we will use Lemma \ref{nestedPartitionDetails}, Equation \ref{multires}, and multinomial concentration (the technical tool stated in Result \ref{multinomialConcentration}).

\frequentistDwasser*
\begin{proof}
Applying Equation \ref{multires}, we have that
\begin{equation}
\label{firstMultiRes}
\mathbb{E}_{p_0} W_v^v(P_0,\bar{P}_n) \lesssim \\ \left(\frac{1}{2}\right)^{K_n v} +\sum_{k=1}^{K_n} \left(\frac{1}{2}\right)^{(k-1) v} \mathbb{E}_{p_0} \sum_{S \in \mathcal{S}_k} |\bar{P}_n(S) - P_0(S)|. 
\end{equation}
Now for $\pmb{j} = (j_1,j_2,\dots,j_d) \in [b_n]^d$, recall that 
\[A_{\pmb{j},b_n} = \left[\frac{j_1-1}{b_n},\frac{j_1}{b_n} \right) \times \left[\frac{j_2-1}{b_n},\frac{j_2}{b_n}\right) \times \dots \times \left[\frac{j_d-1}{b_n},\frac{j_d}{b_n}\right)\] and that $b_n = 2^{\ceil{\log_2(k_n)}} = 2^{K_n}$ and that $\mathcal{S}_{K_n} = \{ A_{\pmb{j},b_n} \}_{\pmb{j} \in [b_n]^d}$. \\


%
By Lemma \ref{nestedPartitionDetails} part (3), we have that for each $k \in \{1,2,\dots,K_n\}$, and each $S \in \mathcal{S}_k$, there is a set $I_{S,k,n} \subseteq [b_n]^d$, such that
\[S = \bigcup_{\pmb{j} \in I_{S,k,n}} A_{\pmb{j},b_n}.\] Additionally, by Lemma \ref{nestedPartitionDetails} part (3), $\{ \bigcup_{\pmb{j} \in I_{S,k,n}} A_{\pmb{j},b_n} \}_{S \in \mathcal{S}_k}$ partitions $[0,1)^d$ and $\{I_{S,k,n}\}_{S \in \mathcal{S}_k}$ partitions $[b_n]^d$. Using these equivalent forms for the $S$ sets and Equation \ref{firstMultiRes}, we have that
\begin{equation}
\label{secondMultiRes}
\mathbb{E}_{p_0} W_v^v(P_0,\bar{P}_n) \lesssim \left(\frac{1}{2}\right)^{K_n v} +\sum_{k=1}^{K_n} \left(\frac{1}{2}\right)^{(k-1) v} \mathbb{E}_{p_0} \sum_{S \in \mathcal{S}_k} \left|\bar{P}_n(\bigcup_{\pmb{j} \in I_{S,k,n}} A_{\pmb{j},b_n}) - P_0(S) \right|. 
\end{equation}
Now by definition of $\bar{P}_n$ (see Equation \ref{posteriorMeanHistogram}) note that for $\pmb{j} \in [b_n]^d$, $\bar{P}_n(A_{\pmb{j},b_n}) = \frac{\alpha_{\pmb{j},b_n}^*}{\sum_{\pmb{i} \in [b_n]^d} \alpha_{\pmb{i},b_n}^*} = \frac{\alpha_{\pmb{j},b_n} + \sum_{t=1}^{n} \mathbb{I}(Y_t \in A_{\pmb{j},b_n})}{n+ \sum_{\pmb{i} \in [b_n]^d} \alpha_{\pmb{i},b_n}}$ and recall the $A_{\pmb{j},b_n}$ are disjoint. Therefore, for each $k \in \{1,2,\dots,K_n\}$, and $S \in \mathcal{S}_k$, we have 
{\allowdisplaybreaks
\begin{align} \label{innerSummand}
\left|\bar{P}_n(\bigcup_{\pmb{j} \in I_{S,k,n}} A_{\pmb{j},b_n}) - P_0(S)\right| & = 
\left|\sum_{\pmb{j} \in I_{S,k,n}}  \frac{\alpha_{\pmb{j},b_n} + \sum_{t=1}^{n} \mathbb{I}(Y_t \in A_{\pmb{j},b_n})}{n+ \sum_{\pmb{i} \in [b_n]^d} \alpha_{\pmb{i},b_n}} - P_0(S) \right| \nonumber \\
& = \left| \frac{\sum_{\pmb{j} \in I_{S,k,n}} \alpha_{\pmb{j},b_n} + \sum_{t=1}^{n} \mathbb{I}(Y_t \in \bigcup_{\pmb{j} \in I_{S,k,n}} A_{\pmb{j},b_n})}{n + \sum_{\pmb{i} \in [b_n]^d} \alpha_{\pmb{i},b_n} } - P_0(S)\right| \nonumber \\
& \leq \left|\frac{n}{n+ \sum_{\pmb{i} \in [b_n]^d} \alpha_{\pmb{i},b_n} } \frac{\sum_{t=1}^{n} \mathbb{I}(Y_t \in S)}{n} - P_0(S)\right| \nonumber \\
&\phantom{{}=aaaaa} + \frac{\sum_{\pmb{j} \in I_{S,k,n}} \alpha_{\pmb{j},b_n} }{ n+\sum_{\pmb{i} \in [b_n]^d} \alpha_{\pmb{i},b_n} } \nonumber \\
& \leq \left|\frac{n}{n+ \sum_{\pmb{i} \in [b_n]^d} \alpha_{\pmb{i},b_n} } -1\right| \frac{\sum_{t=1}^{n} \mathbb{I}(Y_t \in S)}{n} + \nonumber \\
& \phantom{{}=aaaaa}+\left|\frac{\sum_{t=1}^{n} \mathbb{I}(Y_t \in S)}{n} - P_0(S)\right| + \frac{\sum_{\pmb{j} \in I_{S,k,n}} \alpha_{\pmb{j},b_n} }{n+ \sum_{\pmb{i} \in [b_n]^d} \alpha_{\pmb{i},b_n} }.
\end{align}
}

For the first term appearing in the last line of Equation \ref{innerSummand}, recall again that $\mathcal{S}_k$ partitions $[0,1)^d$ and $P_0$ is on $[0,1)^d$. Therefore
\begin{align}
\label{firstTermWassFreq}
\mathbb{E}_{p_0} \sum_{S \in \mathcal{S}_k} \left|\frac{n}{n + \sum_{\pmb{i} \in [b_n]^d} \alpha_{\pmb{i},b_n} } - 1\right| \frac{\sum_{t=1}^{n} \mathbb{I}(Y_t \in S) }{n} & = \frac{\sum_{\pmb{i} \in [b_n]^d} \alpha_{\pmb{i},b_n} }{n + \sum_{\pmb{i} \in [b_n]^d} \alpha_{\pmb{i},b_n} } \sum_{S \in \mathcal{S}_k} P_0(S) \nonumber \\
& = \frac{\sum_{\pmb{i} \in [b_n]^d} \alpha_{\pmb{i},b_n} }{n + \sum_{\pmb{i} \in [b_n]^d} \alpha_{\pmb{i},b_n} }. 
\end{align}
For the second term appearing in the last line of Equation \ref{innerSummand}, recall again that $\{ S \}_{S \in \mathcal{S}_k}$ partitions $[0,1)^d$. Therefore, $\left( \sum_{t=1}^{n} \mathbb{I}(Y_t \in S)\right)_{S \in \mathcal{S}_k} \sim \Mult\left(n, P_0(S) \right)_{S \in \mathcal{S}_k}$. So applying the multinomial concentration Result \ref{multinomialConcentration}, we have that
\begin{equation}
\label{secondTermWassFreq}
\mathbb{E}_{p_0} \sum_{S \in \mathcal{S}_k} \left|\frac{\sum_{t=1}^{n} \mathbb{I}(Y_t \in S)}{n} - P_0(S)\right| \leq n^{-1/2} \sqrt{|\mathcal{S}_k|}.
\end{equation}
For the third term appearing in the last line of Equation \ref{innerSummand}, recall again that $\{I_{S,k,n} \}_{S \in \mathcal{S}_k}$ partitions $[b_n]^d$, and therefore
\begin{equation}
\label{thirdTermWassFreq}
\mathbb{E}_{p_0} \sum_{S \in \mathcal{S}_k} \frac{\sum_{\pmb{j} \in I_{S,k,n}} \alpha_{\pmb{j},b_n} }{n+ \sum_{\pmb{i} \in [b_n]^d} \alpha_{\pmb{i},b_n} } = \frac{\sum_{\pmb{i} \in [b_n]^d} \alpha_{\pmb{i},b_n} }{n + \sum_{\pmb{i} \in [b_n]^d} \alpha_{\pmb{i},b_n}}.
\end{equation}

Using Equations \ref{secondMultiRes}, \ref{innerSummand}, \ref{firstTermWassFreq}, \ref{secondTermWassFreq}, and \ref{thirdTermWassFreq}, we have that
\begin{equation}
\label{thirdMultiRes}
\mathbb{E}_{p_0} W_v^v(P_0,\bar{P}_n) \lesssim \left(\frac{1}{2}\right)^{K_n v} + \frac{\sum_{\pmb{i} \in [b_n]^d} \alpha_{\pmb{i},b_n} }{n + \sum_{\pmb{i} \in [b_n]^d} \alpha_{\pmb{i},b_n}} \sum_{k=1}^{K_n} \left(\frac{1}{2}\right)^{(k-1)v} + n^{-1/2} \sum_{k=1}^{K_n} \left(\frac{1}{2}\right)^{(k-1)v} \sqrt{|\mathcal{S}_k|}.
\end{equation}
Using that $K_n = \ceil{\log_2(k_n)}$ to upper bound the first term appearing in Equation \ref{thirdMultiRes}, and that for every $v \geq 1$ the summand appearing in the second term in Equation \ref{thirdMultiRes} is a partial sum of a convergent geometric series to upper bound the second term appearing in Equation \ref{thirdMultiRes} and that $|\mathcal{S}_k| = 2^{kd}$ for $k \in \{1,2,\dots,K_n\}$ to bound the third term appearing in Equation \ref{thirdMultiRes}, we have that
\begin{equation}
\label{fourthMultiRes}
\mathbb{E}_{p_0} W_v^v(P_0,\bar{P}_n) \lesssim k_n^{-v} + \frac{\sum_{\pmb{i} \in [b_n]^d} \alpha_{\pmb{i},b_n} }{n + \sum_{\pmb{i} \in [b_n]^d} \alpha_{\pmb{i},b_n}} + n^{-1/2}\sum_{k=1}^{K_n} 2^{k(\frac{d}{2}-v)}. 
\end{equation}

Now we consider the 3 different cases. In each case we will use that by Jensen's inequality, $\mathbb{E}_{p_0} W_v(P_0,\bar{P}_n) \lesssim (\mathbb{E}_{p_0} W_v^v(P_0,\bar{P}_n) )^{1/v}$ \\

For the first case ($d < 2v$), note the summand of the last term in Equation \ref{fourthMultiRes} is a partial sum of a convergent geometric series. Therefore using that $\sum_{\pmb{j} \in [b_n]^d} \alpha_{\pmb{j},b_n} \lesssim n^{1/2}$, and that $k_n = n^{1/(2v)}$, we have that
\begin{equation}
\mathbb{E}_{p_0} W_v^v(P_0,\bar{P}_n) \lesssim k_n^{-v} +n^{-1/2}\lesssim n^{-1/2}.
\end{equation}

For the second case ($d = 2v$), note the summand of the last term in Equation \ref{fourthMultiRes} is $\leq K_n$. Therefore using that $K_n = \ceil{\log_2(k_n)}$, and that $\sum_{\pmb{j} \in [b_n]^d} \alpha_{\pmb{j},b_n} \lesssim n^{1/2}$, and that $k_n = n^{1/(2v)}$, we have that
\begin{equation}
\mathbb{E}_{p_0} W_v^v(P_0,\bar{P}_n) \lesssim k_n^{-v} +n^{-1/2}+ n^{-1/2} K_n \lesssim n^{-1/2} \log(n).
\end{equation}

For the third case, ($d > 2v$), note first that for any $q > 0$, 
\begin{equation}
\label{convenientGeoSeries}
\sum_{j=1}^{K_n} 2^{k q} \leq \frac{2^{q}}{2^q -1} [(2^q)^{K_n} -1] \lesssim 2^{q K_n}.
\end{equation}
Applying this with $q = \frac{d}{2} -v$ and using that $k_n = n^{1/d}$ and that $\sum_{\pmb{j} \in [b_n]^d} \alpha_{\pmb{j},b_n} \lesssim n^{1-\frac{v}{d}}$ and that $K_n = \ceil{\log_2(k_n)}$, we have that
\begin{equation}
\mathbb{E}_{p_0} W_v^v(P_0,\bar{P}_n) \lesssim k_n^{-v} + n^{-\frac{v}{d}} + n^{-1/2} 2^{(\frac{d}{2} - v) K_n} \lesssim n^{-v/d} +n^{-1/2} n^{(\frac{1}{d}) (\frac{d}{2} -v)} \lesssim n^{-\frac{v}{d}}.
\end{equation}

\end{proof}

\subsubsection{Theorem \ref{contractionAroundCentralEstimatorDwasser}}

We now intend to prove a statement about how the sequence of posterior distributions contracts around the posterior mean histogram sequence $\bar{P}_n$. To do so, we will be using the concentration of the Dirichlet distribution in the $\| \cdot \|_{1}$ distance around its mean (see Result \ref{diricheletConcentration}). To use this concentration the following preliminary Lemma will be helpful.

\begin{lemma}
\label{closeInParamHelper}
Define 
\[
k_n := \begin{cases}
n^{\frac{1}{2v}} & d \leq 2v. \\
n^{\frac{1}{d}} & d > 2v.
\end{cases}
\]
For $n \in \mathbb{N}$, $k \in \{1,2,\dots,K_n\}$ and $S \in \mathcal{S}_k$, let $I_{S,k,n}$ be defined as in the proof of Lemma \ref{nestedPartitionDetails} part 3. For $\gamma > 1$, if $\pmb{\pi}_1,\pmb{\pi}_2 \in \mathcal{S}^{b_n d-1}$ and for each $k \in \{1,2,\dots,K_n\}$, 
\[\sum_{S \in \mathcal{S}_k} | \sum_{\pmb{j} \in I_{S,k,n}} \pi_{1\pmb{j}} - \sum_{\pmb{j} \in I_{S,k,n}} \pi_{2\pmb{j}} | \leq \log^{\gamma}(n) \sqrt{\frac{2^{dk} }{n+\sum_{\pmb{j} \in [b_n]^d} \alpha_{\pmb{j},b_n}}} \]
and if $\alpha_{\pmb{j},b_n} > 0$ for each $\pmb{j} \in [b_n]^d$, then
 \begin{equation}
 \sum_{k=1}^{K_n} (\frac{1}{2})^{k v} \sum_{S \in \mathcal{S}_k} | \sum_{\pmb{j} \in I_{S,k,n}} \pi_{1\pmb{j}} - \sum_{\pmb{j} \in I_{S,k,n}} \pi_{2\pmb{j}} | \leq 
 \begin{cases}
 C_1(d,v) n^{-\frac{1}{2}} \log^{\gamma}(n)& d < 2v, \\
C_2(d,v) n^{-\frac{1}{2}} \log^{1+\gamma}(n)	& d = 2v,\\
 C_3(d,v) \log^{\gamma}(n) n^{-\frac{v}{d}}	& d > 2v,
 \end{cases}
 \end{equation}
 where
 \begin{align*}
C_1(d,v) &\geq \frac{2^{(\frac{d}{2}-v)}}{1-2^{(\frac{d}{2}-v)}}, &
C_2(d,v) &\geq \frac{2}{d} \frac{1}{\log(2)}, &
C_3(d,v) \geq \frac{2^{2(\frac{d}{2}-v)}}{2^{(\frac{d}{2}-v)}-1}.
 \end{align*}
\end{lemma}
\begin{proof}
By the assumed closeness for each $k \in \{1,2,\dots,K_n\}$ we have that
\begin{equation}
\label{firstUpper}
 \sum_{k=1}^{K_n} (\frac{1}{2})^{k v} \sum_{S \in \mathcal{S}_k} | \sum_{\pmb{j} \in I_{S,k,n}} \pi_{1\pmb{j}} - \sum_{\pmb{j} \in I_{S,k,n}} \pi_{2\pmb{j}} | \leq
 n^{-1/2} \log^{\gamma}(n) \sum_{k=1}^{K_n} 2^{k(\frac{d}{2}-v)}.
\end{equation}
For the case $d < 2v$, $\sum_{k=1}^{K_n} 2^{k(\frac{d}{2}-v)} \leq  \frac{2^{(\frac{d}{2}-v)}}{1-2^{(\frac{d}{2}-v)}}$. \\
For the case $d = 2v$, $\sum_{k=1}^{K_n} 2^{k (\frac{d}{2}-v)} = K_n \leq \frac{2}{d} \log_2(n) = \frac{2}{d} \frac{1}{\log(2)} \log(n)$.
For the case $d > 2v$, by Equation \ref{convenientGeoSeries}
\[\sum_{k=1}^{K_n} 2^{k(\frac{d}{2}-v)} \leq \frac{2^{(\frac{d}{2}-v)}}{2^{(\frac{d}{2}-v)}-1} (2^{K_n(\frac{d}{2} -v)}-1) \leq \frac{2^{2(\frac{d}{2}-v)}}{2^{(\frac{d}{2}-v)}-1} n^{\frac{1}{d} (\frac{d}{2}-v)}.\] Using this and that $-\frac{1}{2} + \frac{1}{d}(\frac{d}{2} - v) = -\frac{d}{v}$ completes the case.
\end{proof}
Now we are ready to prove contraction around the posterior mean histogram $\bar{P}_n$.

\contractionAroundCentralEstimatorDwasser*
\begin{proof}
For $n \in \mathbb{N}$, $\pmb{\pi}_1,\pmb{\pi}_2 \in \mathcal{S}^{b_n d-1}$, setting $K = K_n$ in Equation \ref{multires}, we have that
\begin{equation}
W_v^v(\psi_{b_n}(\pmb{\pi}_1),\psi_{b_n}(\pmb{\pi}_2)) \leq d^{v/p} \left[\left(\frac{1}{2}\right)^{K_n} + \sum_{k=1}^{K_n} \left(\frac{1}{2}\right)^{(k-1)v} \sum_{S \in \mathcal{S}_k} | \psi_{b_n}(\pmb{\pi}_1)(S) - \psi_{b_n}(\pmb{\pi}_2)(S)|\right].
\end{equation}
Now using Lemma \ref{nestedPartitionDetails} part 3, we further simplify the above upper bound to
{\allowdisplaybreaks
\begin{equation}
\begin{aligned}
\label{temporaryUpper}
W_v^v(\psi_{b_n}(\pmb{\pi}_1),\psi_{b_n}(\pmb{\pi}_2))  \\ & \leq
d^{v/p} \left[\left(\frac{1}{2}\right)^{K_nv} + \sum_{k=1}^{K_n} \left(\frac{1}{2}\right)^{(k-1)v} \sum_{S \in \mathcal{S}_k} \left|\psi_{b_n}(\pmb{\pi}_1) \left(\bigcup_{\pmb{j} \in I_{S,k,n} } A_{\pmb{j},b_n}\right) - \right. \right.\\
& \left. \left. \phantom{aaaa} \psi_{b_n}(\pmb{\pi}_2) \left(\bigcup_{\pmb{j} \in I_{S,k,n}} A_{\pmb{j},b_n} \right) \right| \right] \\ & \leq
 d^{v/p} \left[\left(\frac{1}{2}\right)^{K_nv} + \sum_{k=1}^{K_n} \left(\frac{1}{2}\right)^{(k-1)v} \sum_{S \in \mathcal{S}_k} \left| \sum_{\pmb{j} \in I_{S,k,n}} \psi_{b_n}(\pmb{\pi}_1)(A_{\pmb{j},b_n}) - \right. \right. \\
& \left. \left. \phantom{aaaa} \sum_{\pmb{j} \in I_{S,k,n}} \psi_{b_n}(\pmb{\pi}_2)(A_{\pmb{j},b_n}) \right| \right] \\ & \leq
 d^{v/p} \left[\left(\frac{1}{2}\right)^{K_nv} + \sum_{k=1}^{K_n} \left(\frac{1}{2}\right)^{(k-1)v} \sum_{S \in \mathcal{S}_k} \left| \sum_{\pmb{j} \in I_{S,k,n}} \pi_{1\pmb{j}} - \sum_{\pmb{j} \in I_{S,k,n}} \pi_{2\pmb{j}} \right| \right], 
\end{aligned}
\end{equation}
}
where to get the first inequality of Equation \ref{temporaryUpper} we use lemma \ref{nestedPartitionDetails} part 3. For the second inequality we use that $\{A_{\pmb{j},b_n} \}_{\pmb{j} \in [b_n]^d}$ partitions $[0,1)^d$ and in particular the $A_{\pmb{j},n}$ sets are disjoint. To get the third inequality, we use that by definition of $\psi_{b_n}$, for every $\pmb{\pi} \in \mathcal{S}^{b_n d-1}$ and every $\pmb{j} \in [b_n]^d$, $\psi_{b_n}(\pmb{\pi})(A_{\pmb{j},b_n}) = \pi_{\pmb{j}}$.

Now using Equation \ref{temporaryUpper}, the preimage form of $\Pi_n$ (see Equation \ref{preImageForm}), the definition of $\bar{P}_n$ (see Equation \ref{posteriorMeanHistogram}), and the definition of $z_n^*$ (the posterior measure over the simplex $\mathcal{S}^{b_n d-1}$), we have that for any $1 \leq v < \infty , d \in \mathbb{N}$, almost surely under $P_0$ and eventually in $n$,
{\allowdisplaybreaks
\begin{equation}
\begin{aligned}
\label{settingUpTheUnionBound}
& \Pi_n(P \in \mathcal{P}_d: W_v(P,\bar{P}_n) \geq \tau_n(d,v)) \\  
& = z_n^*(\pmb{\pi}_1 \in \mathcal{S}^{b_n d-1}: W_v^v(\psi_{b_n}(\pmb{\pi}_1),\psi_{b_n}(\mathbb{E}_{z_n^*}(\pmb{\pi} | Y_1,\dots,Y_n) ) ) \geq \tau_n^{v}(d,v) ) \\
& \leq  z_n^* \left(\pmb{\pi}_1 \in \mathcal{S}^{b_n d-1}: \frac{1}{2^{K_n v}} + \sum_{k=1}^{K_n} \frac{1}{2^{(k-1)v}} \sum_{S \in \mathcal{S}_k} \left| \sum_{\pmb{j} \in I_{S,k,n}} \pi _{1\pmb{j}} - \right. \right. \\
& \hspace{5cm} \left. \left. \sum_{\pmb{j} \in I_{S,k,n}} \mathbb{E}_{z_n^*} (\pi_{\pmb{j}} | Y_1,\dots,Y_n) \right| \geq d^{-v/p}\tau_n^{v}(d,v) \right) \\
& \leq z_n^*\left(\pmb{\pi}_1 \in \mathcal{S}^{b_n d-1}:  \sum_{k=1}^{K_n} \frac{1}{2^{(k-1)v}} \sum_{S \in \mathcal{S}_k} \left| \sum_{\pmb{j} \in I_{S,k,n}} \pi_{1\pmb{j}} - \sum_{\pmb{j} \in I_{S,k,n}} \mathbb{E}_{z_n^*} (\pi_{\pmb{j}} | Y_1,\dots,Y_n) \right| \right. \\
& \hspace{5cm} \left. \geq \frac{1}{2} d^{-v/p}\tau_n^{v}(d,v) \right) \\
& \leq z_n^*\left(\pmb{\pi}_1 \in \mathcal{S}^{b_n d-1}:  \sum_{k=1}^{K_n} \frac{1}{2^{kv}} \sum_{S \in \mathcal{S}_k} \left| \sum_{\pmb{j} \in I_{S,k,n}} \pi_{1\pmb{j}} - \sum_{\pmb{j} \in I_{S,k,n}} \mathbb{E}_{z_n^*} (\pi_{\pmb{j}} | Y_1,\dots,Y_n) \right| \right. \\
& \hspace{5cm}  \geq 2^{v-1} d^{-v/p}\tau_n^{v}(d,v) \Bigg),
\end{aligned}
\end{equation} }
where the second last inequality is eventually in $n$, using that in all three cases ($(d <2v, k_n = n^{1/2v}), (d = 2v, k_n = n^{1/2v}), (d > 2v, k_n = n^{1/d}$), $(1/2)^{K_n v}  \leq \frac{1}{2} d^{-v/p} \tau_n^v(d,v)$ eventually in $n$. This is because in each case $(1/2)^{K_n v} \leq k_n^{-v}$ and $\tau_n^{v}(d,v) \gtrsim k_n^{-v} \log^{\omega}(n)$ for some $\omega > 0$ where the value of $\omega$ depends on the case.

Now note that for $d < 2v$ since $C_4(d,v) > \frac{1}{2}C_1(d,v)^{1/v} 2^{1/v} d^{1/p} $, we have that
\begin{equation}
\label{justifyingClosenessCase1}
C_1(d,v) n^{-1/2} \log^{\gamma}(n) < 2^{v-1} d^{-v/p} \tau_n^{v}(d,v),
\end{equation}
and note that for $d = 2v$, since $C_5(d,v) > \frac{1}{2} C_2(d,v)^{1/v} 2^{1/v} d^{1/p}$, we have that
\begin{equation}
C_2(d,v) n^{-1/2} \log^{1+\gamma}(n)  < 2^{v-1} d^{-v/p} \tau_n^{v}(d,v),
\label{justifyingClosenessCase2}
\end{equation}
and note that for $d > 2v$, since $C_6(d,v) > \frac{1}{2} C_3(d,v)^{1/v} 2^{1/v} d^{1/p}$, we have that
\begin{equation}
\label{justifyingClosenessCase3}
C_3(d,v) n^{-\frac{v}{d}} \log^{\gamma}(n) < 2^{v-1} d^{-v/p} \tau_n^{v}(d,v).
\end{equation}
So applying Lemma \ref{closeInParamHelper} with the value $\gamma$ using Equation \ref{justifyingClosenessCase1} when $d < 2v$, Equation \ref{justifyingClosenessCase2} when $d = 2v$ and Equation \ref{justifyingClosenessCase3} we have that for all $v \geq 1$ and $d \in \mathbb{N}$,
{\allowdisplaybreaks
\begin{align}
\label{connectingToUnionBound}
& z_n^* \left(\pmb{\pi}_1 \in \mathcal{S}^{b_n d-1}:  \sum_{k=1}^{K_n} \frac{1}{2^{kv}} \sum_{S \in \mathcal{S}_k} \left| \sum_{\pmb{j} \in I_{S,k,n}} \pi_{1\pmb{j}} - \sum_{\pmb{j} \in I_{S,k,n}} \mathbb{E}_{z_n^*} (\pi_{\pmb{j}} | Y_1,\dots,Y_n) \right| \geq 2^{v-1} d^{-v/p}\tau_n^{v}(d,v) \right) \nonumber \\
& \leq z_n^* \left(\pmb{\pi}_1 \in \mathcal{S}^{b_n d-1}: \exists k \in \{1,2,\dots,K_n\} \text{ s.t }  \sum_{S \in \mathcal{S}_k} \left| \sum_{\pmb{j} \in I_{S,k,n}} \pi_{1\pmb{j}} - \sum_{\pmb{j} \in I_{S,k,n}} \mathbb{E}_{z_n^*} (\pi_{\pmb{j}} | Y_1,\dots,Y_n) \right| \right. \nonumber \\
& \hspace{8cm} \left. > \log^{\gamma}(n) \sqrt{ \frac{2^{dk} }{n + \sum_{\pmb{j} \in [b_n]^d} \alpha_{\pmb{j},b_n} } } \right) \nonumber \\ 
& \leq \sum_{k=1}^{K_n} z_n^*\left(\pmb{\pi}_1 \in \mathcal{S}^{b_n d-1}: \sum_{S \in \mathcal{S}_k} \left| \sum_{\pmb{j} \in I_{S,k,n}} \pi_{1\pmb{j}} - \sum_{\pmb{j} \in I_{S,k,n}} \mathbb{E}_{z_n^*} (\pi_{\pmb{j}} | Y_1,\dots,Y_n) \right| \right. \nonumber \\ 
& \hspace{8cm} \left. >  \log^{\gamma}(n) \sqrt{ \frac{2^{dk} }{n + \sum_{\pmb{j} \in [b_n]^d} \alpha_{\pmb{j},b_n} } } \right),
\end{align}
}
where in the last line we have used the union bound.

Now note that for $v \geq 1, d \in \mathbb{N}$, by Lemma \ref{nestedPartitionDetails} part 3, $\{I_{S,k,n}\}_{S \in \mathcal{S}_k}$ partitions $[b_n]^d$ for $k \in \{1,2,\dots,K_n\}$. In particular, since  $z_n^* = \Dir(\cdot| \{\alpha_{\pmb{j},b_n}^{*}\}_{\pmb{j} \in [b_n]^d})$, under $z_n^*$, $\{\sum_{\pmb{j} \in I_{S,k,n}} \pi_{\pmb{j}} \}_{S \in \mathcal{S}_k} \sim \Dir( \{\sum_{\pmb{j} \in I_{S,k,n}} \alpha_{\pmb{j},b_n}^*\}_{S \in \mathcal{S}_k})$. Moreover, $\sum_{S \in \mathcal{S}_k} \sum_{\pmb{j} \in I_{S,k,n}} \alpha_{\pmb{j},b_n}^{*} = \sum_{\pmb{j} \in [b_n]^d} \alpha_{\pmb{j},b_n}^{*} \overset{a.s}{=} n + \sum_{\pmb{j} \in [b_n]^d} \alpha_{\pmb{j},b_n}$. Finally note that by definition of $\mathcal{S}_k$, $|\mathcal{S}_k| = 2^{dk}$. So for $n \in \mathbb{N}$ and $k \in \{1,2,\dots,K_n\}$ applying Dirichlet concentration of measure Equation \ref{diricheletConcentration} with $\delta := \log^{-\gamma}(n)$ , we have that
\begin{align}
\label{applyingDiricheletConcentration}
& z_n^*\left(\pmb{\pi}_1 \in \mathcal{S}^{b_n d-1}: \sum_{S \in \mathcal{S}_k} \left| \sum_{\pmb{j} \in I_{S,k,n}} \pi_{1\pmb{j}} - \sum_{\pmb{j} \in I_{S,k,n}} \mathbb{E}_{z_n^*} (\pi_{\pmb{j}} | Y_1,\dots,Y_n) \right| > \log^{\gamma}(n) \sqrt{ \frac{2^{dk} }{n + \sum_{\pmb{j} \in [b_n]^d} \alpha_{\pmb{j},b_n} } } \right) \nonumber \\ 
& \leq \log^{-\gamma}(n).
\end{align}
By Equations \ref{connectingToUnionBound} and \ref{applyingDiricheletConcentration}, we have that for $v \geq 1$, $d \in \mathbb{N}$,$n \in \mathbb{N}$ and almost surely under $P_0$
\begin{align}
\label{penultimate}
& z_n^*\left(\pmb{\pi}_1 \in \mathcal{S}^{b_n d-1}:  \sum_{k=1}^{K_n} (1/2)^{kv} \sum_{S \in \mathcal{S}_k} \left| \sum_{\pmb{j} \in I_{S,k,n}} \pi_{1\pmb{j}} - \sum_{\pmb{j} \in I_{S,k,n}} \mathbb{E}_{z_n^*} (\pi_{\pmb{j}} | Y_1,\dots,Y_n)\right| \geq 2^{v-1} d^{-v/p}\tau_n^{v}(d,v) \right) \nonumber \\
& \leq \sum_{k=1}^{K_n} \log^{-\gamma}(n) \nonumber \\
& = K_n \log^{-\gamma}(n).
\end{align}
Since $\gamma > 1$ and $K_n \log^{-\gamma}(n) \lesssim \log^{1-\gamma}(n)$ for every $v \geq 1, d \in \mathbb{N}$, we conclude that
\begin{align}
& z_n^*\left(\pmb{\pi}_1 \in \mathcal{S}^{b_n d-1}:  \sum_{k=1}^{K_n} (1/2)^{kv} \sum_{S \in \mathcal{S}_k} \left| \sum_{\pmb{j} \in I_{S,k,n}} \pi_{1\pmb{j}} - \sum_{\pmb{j} \in I_{S,k,n}} \mathbb{E}_{z_n^*} (\pi_{\pmb{j}} | Y_1,\dots,Y_n)\right| \geq 2^{v-1} d^{-v/p}\tau_n^{v}(d,v) \right) \nonumber \\
& \to 0
\end{align}
as $n \to \infty$ almost surely under $P_0$.

Using this and Equation \ref{settingUpTheUnionBound} we conclude that for every $v \geq 1, d \in \mathbb{N}$, almost surely under $P_0$
\begin{equation}
\label{ultimate}
\Pi_n(P \in \mathcal{P}_d: W_v(P,\bar{P}_n) \geq \tau_n(d,v)) \to 0 \text{ as } n \to \infty.
\end{equation}
By dominated convergence the conclusion of the theorem follows.
\end{proof}

\subsubsection{Theorem \ref{concludingTheoremDwasser}}

We are now finally ready to state and easily prove posterior contraction rates using Theorems \ref{frequentistDwasser} and \ref{contractionAroundCentralEstimatorDwasser}.

\concludingTheoremDwasser*
\begin{proof}
By the triangle inequality and the union bound
\begin{align}
\label{dWasserGeneralStrategy}
\mathbb{E}_{p_0} \left[ \Pi_n \left(P \in \mathcal{P}_d : W_v(P_0,P) \geq \epsilon_n(d,v) \right) \right] & \leq
\mathbb{E}_{p_0} \left[\Pi_n \left(P \in \mathcal{P}_d: W_v(P_0,\bar{P}_n) \geq \frac{\epsilon_n(d,v)}{2}\right)\right] \nonumber \\
&  \phantom{aaaaa} +\mathbb{E}_{p_0} \left[\Pi_n \left(P \in \mathcal{P}_d: W_v(P,\bar{P}_n) \geq \frac{\epsilon_n(d,v)}{2} \right) \right] \nonumber \\
& = P_0 \left[W_v(P_0,\bar{P}_n) \geq \frac{\epsilon_n(d,v)}{2} \right] \nonumber \\ 
& \phantom{aaaaa} + \mathbb{E}_{p_0} \left[\Pi_n\left(P \in \mathcal{P}_d: W_v(P,\bar{P}_n) \geq \frac{\epsilon_n(d,v)}{2}\right)\right],
\end{align}
\sloppy where the equality in the last line is because $\Pi_n\left(P \in \mathcal{P}_d: W_v(P_0,\bar{P}_n) \geq \frac{\epsilon_n(d,v)}{2}\right) = \mathbb{I}\left(W_v(P_0,\bar{P}_n) \geq \frac{\epsilon_n(d,v)}{2}\right)$. Using Markov's inequality and Theorem \ref{frequentistDwasser}
\begin{equation}
\label{usingFrequentistRate}
P_0\left[W_v(P_0,\bar{P}_n) \geq \frac{\epsilon_n(d,v)}{2} \right] \leq 2 \frac{\mathbb{E}_{p_0} W_v(P_0,\bar{P}_n)}{\epsilon_n(d,v)} \lesssim
\begin{rcases}
	\begin{dcases}
\frac{n^{-\frac{1}{2v}}}{n^{-\frac{1}{2v}}  \log^{\frac{\gamma}{v} }(n) }& d < 2v \\
\frac{n^{-\frac{1}{2v}} \log^{\frac{1}{v}} (n)}{n^{-\frac{1}{2v}} \log^{\frac{1+\gamma}{v}} (n)}& d = 2v \\
\frac{n^{-\frac{1}{d}}}{n^{-\frac{1}{d}} \log^{\frac{\gamma}{v}}(n) }& d > 2v
	\end{dcases}
\end{rcases} \to 0 \text{ as } n \to \infty.
\end{equation}
Since $C_7(d,v) \geq 2 C_4(d,v)$, $C_8(d,v) \geq 2 C_5(d,v)$, $C_9(d,v) \geq 2 C_6(d,v)$, we have that $\tau_n(d,v) \leq \frac{\epsilon_n(d,v)}{2}$ for every $v \geq 1, d \in \mathbb{N}$ where $\tau_{n}(d,v)$ is as defined in Theorem \ref{contractionAroundCentralEstimatorDwasser}. Using this and Theorem \ref{contractionAroundCentralEstimatorDwasser}, we have that for every $v \geq 1, d \in \mathbb{N}$,
\begin{equation}
\label{usingContractionRate}
\mathbb{E}_{p_0} \left[\Pi_n\left(P \in \mathcal{P}_d: W_v(P,\bar{P}_n) \geq \frac{\epsilon_n(d,v)}{2}\right)\right] \leq \mathbb{E}_{p_0} \left[\Pi_n\left(P \in \mathcal{P}_d: W_v(P,\bar{P}_n) \geq \tau_n(d,v)\right)\right] \to 0
\end{equation}
as $n \to \infty$. By Equations \ref{dWasserGeneralStrategy}, \ref{usingFrequentistRate}, and \ref{usingContractionRate}, we conclude that for all $d \in \mathbb{N}, v \geq 1$,
\begin{equation}
\mathbb{E}_{p_0} \left[\Pi_n\left(P \in \mathcal{P}_d : W_v(P_0,P) \geq \epsilon_n(d,v) \right) \right] \to 0
\end{equation}
as $n \to \infty$. By Markov the theorem statement follows.
\end{proof}

\section{Connection between the dyadic Bayes histogram and the \cite{niles2022minimax} histogram}
\label{app:nwbConnect}

In this section we rigorously show that the histogram presented in \citet[Theorem~6]{niles2022minimax}, which we denote $\ddot{P}_n$, is also a dyadic histogram, but with the prior concentrations all $0$. Furthermore, we observe here that the restriction on the resolution parameter used in the proof of \citet[Theorem 6]{niles2022minimax} implies that $\ddot{P}_n$ has $\gtrsim n^{d/(d+s)}$ bins where $s$ is the assumed holder regularity of the density to be estimated.

To begin studying $\ddot{P}_n$, we first recall the form of the $d$ dimensional Haar wavelet basis on $[0,1)^d$. Specifically, the father wavelet of the Haar system, denoted $\phi_{F}$, is defined as the constant function on $[0,1)$. The mother wavelet, denoted $\psi_{M}$, is defined as 
\begin{equation}
    \label{motherWavelet}
    \psi_M(x) := \begin{cases} 1 & 0 \leq x < 1/2. \\
                              -1 & 1/2 \leq x < 1. \\
                              0 & o.w. \\
                \end{cases}
\end{equation} 
For $u \in \{0,1,2,\dots\}$, the Haar wavelet basis on the space $[0,1)^d$ at resolution $u$ consists of the functions on $[0,1)^d$
\begin{equation}
    \Gamma_{G,m}^{u}(x) := 2^{\frac{dj}{2}} \prod_{l=1}^{d} \psi_{G_l}(2^u x_l - m_l), \; G \in \{F,M\}^{d*}, \; m \in \{0,1,2,\dots,2^u-1\}^{d},
\end{equation}

where $\{F,M\}^{d*}$ is the collection of $d$ dimensional ordered sets consisting only of the objects $F$ and $M$ \text{and} excluding the object that consists only of $F$ repeated $d$ times. \cite{niles2022minimax} use the notation $\Psi_{u}$ to refer to this collection of functions. Specifically
\begin{equation}
    \Psi_{u} := \{ \Gamma^{u}_{G,m} : G \in \{F,M\}^{d*}, m \in \{0,1,2,\dots,2^u-1\}^{d} \}.
\end{equation}

For a given sample size $n$, a collection of estimators indexed by a resolution parameter $J \in \{0,1,2,\dots,\}$ are introduced in Theorem 6 of \cite{niles2022minimax}. The measure of these estimators we denote by $\ddot{P}_{n,J}$; the density is denoted by $\hat{f}_{n,J}$. $\hat{f}_{n,J}$ is defined as
\begin{equation}
    \label{nwbHistDef}
    \hat{f}_{n,J} := 1+\sum_{0 \leq u < J} \sum_{\Gamma \in \Psi_u} \hat{\beta}_{\Gamma,n} \Gamma,
\end{equation}
where for $\Gamma \in \Psi_u$
\begin{equation}
    \label{nwbBetaDef}
\hat{\beta}_{\Gamma,n} := \frac{1}{n} \sum_{i=1}^{n} \Gamma(Y_i).
\end{equation}

The estimator analyzed in \citet[Theorem~6]{niles2022minimax} is $\hat{f}_{n, J = J(n)}$ where $J(n)$ is a function of $n$ that must satisfy $n \asymp 2^{J(n)(d+s)}$ where $s$ is the $0 < s < 1$ holder smoothness of the density to be estimated. We denote $\hat{f}_{n, J = J(n)}$ by $\ddot{p}_n$ (and its corresponding probability measure by $\ddot{P}_n$).

We now prove that $\ddot{P}_n$ is a dyadic histogram with zero prior concentration and clearly establish the relation between $J(n)$ and the number of bins in this histogram. We will do this in the 2 dimensional case only; the notation and geometry is more cumbersome in arbitrarily high dimension $d$, but the main proof techniques are illustrated by the $d=2$ case. For this purpose define the two dimensional dyadic histogram classes of densities $\{\mathcal{H}_{\ell} \}_{\ell \in \{0,1,2,\dots\}}$ where
\begin{equation}
    \label{base2HistoClasses}
    \mathcal{H}_{\ell} := \{ p_{\pmb{\pi}} = \HistDens(\cdot|\pmb{\pi},2^{\ell}), \pmb{\pi} \in \mathcal{S}^{2^{2 \ell}-1} \}
\end{equation}
where recall
\begin{equation}
    \label{base2HistoDefinition}
p_{\pmb{\pi}}(x) :=  \sum_{(i_1,i_2) \in [2^{\ell}]^2} 2^{2 \ell} \pi_{(i_1,i_2)}\mathbb{I}(x \in A_{(i_1,i_2),2^{\ell}}),
\end{equation}
To help make notation concise in this section, for $\ell \in \{0,1,2,\dots\}$ and $(t_1,t_2) \in [2^{\ell}-1]^2$, define $B^{\ell}_{(t_1,t_2)} := [t_1 2^{-\ell}, (t_1+1) 2^{-\ell}) \times [t_2 2^{-\ell},(t_2+1) 2^{-\ell})$. Thus $p_{\pmb{\pi}}$ can be rewritten as
\begin{equation}
    p_{\pmb{\pi}}(x) =  \sum_{(t_1,t_2) \in \{0,1,2,\dots,2^{\ell}-1\}^2 } 2^{2 \ell} \pi_{(t_1+1,t_2+1)}\mathbb{I}(x \in B^{\ell}_{(t_1,t_2)}).
\end{equation}
We now establish that $\ddot{p}_n \in \mathcal{H}_{J(n)}$.
\begin{lemma}
\label{nwbLemma}
Let $n \in \mathbb{N}$. For $J \in \{0,1,2,\dots,\}$, $\hat{f}_{n,J} \in \mathcal{H}_{J}$ with parameter $\pmb{\pi}_J$ and for $(t_1,t_2) \in \{0,1,2,\dots,2^J-1\}^2$, $\pi_{J,(t_1+1,t_2+1)} := \frac{1}{n} \sum_{i=1}^{n} \mathbb{I}(Y_i \in B^{J}_{(t_1,t_2)})$. In particular, $\ddot{p}_n \in \mathcal{H}_{J(n)}$ with parameter $\pmb{\pi}_{J(n)}$.
\end{lemma}
\begin{proof}
    Fix $n \in \mathbb{N}$. The proof is by induction on the quantity $J$. For the base case, suppose $J = 0 = \ceil{\log_2(k_n)}$. Then $\pi_{0,(1,1)} =1$ and $\hat{f}_{n,0} =1 = 2^{0} \pi_{0,(1,1)}$ on $[0,1)^2$. So $\hat{f}_{n,0} \in \mathcal{H}_{0}$ with parameter $\pmb{\pi}_{0}$.

    For the inductive step, suppose that for some $J \in \{0,1,2,\dots,\}$, $\hat{f}_{n,J} \in \mathcal{H}_{J}$ with parameter $\pmb{\pi}_J$ . Note that 
    \begin{equation}
        \label{nwbInductiveConnect}
        \hat{f}_{n,J+1} = \hat{f}_{n,J} + \sum_{\Gamma \in \Psi_J} \hat{\beta}_{\Gamma,n} \Gamma,
    \end{equation}
    and consider the partition of $\{0,1,\dots,2^{J+1}-1\}^2$ that is 
    \[\{S_{(t_1,t_2)}, (t_1,t_2) \in \{0,1,2,\dots,2^{J}-1\}^2 \} \] where
    \begin{equation}
        \label{nwbIndexSets}
        S_{(t_1,t_2)} := \{ (2t_1,2t_2),(2t_1+1,2t_2),(2t_1,2t_2+1),(2t_1+1,2t_2+1) \}.
    \end{equation}
    For $(t_1,t_2) \in \{0,1,2,\dots,2^{J}-1\}^2$, $B_{(t_1,t_2)}^{J} = \bigcup_{(z_1,z_2) \in S_{(t_1,t_2)} } B_{(z_1,z_2)}^{J+1}$. We will now show that on each of the $4$ sets at the $(J+1)^{th}$ resolution corresponding to the indices $(z_1,z_2) \in S_{(t_1,t_2)}$, $\hat{f}_{n,J+1} = 2^{2(J+1)} \pi_{J+1,(z_1+1,z_2+1)}$. First note that there are only $3$ functions in $\Psi_J$ that are non-zero on $B_{(t_1,t_2)}^{J}$. These are $\Gamma^{J}_{\{M,M\},(t_1,t_2)},\Gamma^{J}_{\{M,F\},(t_1,t_2)},\Gamma^{J}_{\{F,M\},(t_1,t_2) }$ and 
    \begin{equation}
        \label{nwbMM}
        \Gamma^{J}_{\{M,M\},(t_1,t_2)} = 2^{\frac{2 J}{2}}\begin{cases}
         1   & x \in B^{J+1}_{(2t_1,2t_2)} \\
         -1   & x \in B^{J+1}_{(2t_1+1,2t_2)} \\
         -1   & x \in B^{J+1}_{(2t_1,2t_2+1)} \\
         1  & x \in B^{J+1}_{(2t_1+1,2t_2+1)}
        \end{cases}
    \end{equation}
    and
    \begin{equation}
        \label{nwbMF}
        \Gamma^{J}_{\{M,F\},(t_1,t_2)} = 2^{\frac{2 J}{2}}\begin{cases}
         1   & x \in B^{J+1}_{(2t_1,2t_2)} \\
         -1   & x \in B^{J+1}_{(2t_1+1,2t_2)} \\
         1   & x \in B^{J+1}_{(2t_1,2t_2+1)} \\
         -1  & x \in B^{J+1}_{(2t_1+1,2t_2+1)}
        \end{cases}
    \end{equation}
    and
    \begin{equation}
        \label{nwbFM}
        \Gamma^{J}_{\{F,M\},(t_1,t_2)} = 2^{\frac{2 J}{2}}\begin{cases}
         1   & x \in B^{J+1}_{(2t_1,2t_2)} \\
         1   & x \in B^{J+1}_{(2t_1+1,2t_2)} \\
         -1   & x \in B^{J+1}_{(2t_1,2t_2+1)} \\
         -1  & x \in B^{J+1}_{(2t_1+1,2t_2+1)}
        \end{cases}
    \end{equation}
    Using Equations \ref{nwbMM},\ref{nwbMF}, and \ref{nwbFM} and that these three functions are the only ones that are non-zero in $\Psi_J$ for $x \in B_{(t_1,t_2)}^{J}$, we have that
    \[\sum_{\Gamma \in \Psi_J} \hat{\beta}_{\Gamma,n} \Gamma(x) = \]
    \begin{equation}
        \label{nwbWaveletPart}   \frac{2^{2J}}{n} \begin{cases}
         3 \sum_{i=1}^{n} \mathbb{I}(Y_i \in B^{J+1}_{(2t_1,2t_2)})- \sum_{i=1}^{n} \mathbb{I}(Y_i \in B^{J+1}_{(2t_1+1,2t_2)})  & x \in B^{J+1}_{(2t_1,2t_2)} \\
         \qquad -\sum_{i=1}^{n} \mathbb{I}(Y_i \in B^{J+1}_{(2t_1,2t_2+1)}) - \sum_{i=1}^{n} \mathbb{I}(Y_i \in B^{J+1}_{(2t_1+1,2t_2+1)}) & \\
         - \sum_{i=1}^{n} \mathbb{I}(Y_i \in B^{J+1}_{(2t_1,2t_2)})+3 \sum_{i=1}^{n} \mathbb{I}(Y_i \in B^{J+1}_{(2t_1+1,2t_2)})   & x \in B^{J+1}_{(2t_1+1,2t_2)} \\
         \qquad -\sum_{i=1}^{n} \mathbb{I}(Y_i \in B^{J+1}_{(2t_1,2t_2+1)}) - \sum_{i=1}^{n} \mathbb{I}(Y_i \in B^{J+1}_{(2t_1+1,2t_2+1)}) & \\
         - \sum_{i=1}^{n} \mathbb{I}(Y_i \in B^{J+1}_{(2t_1,2t_2)})- \sum_{i=1}^{n} \mathbb{I}(Y_i \in B^{J+1}_{(2t_1+1,2t_2)})  & x \in B^{J+1}_{(2t_1,2t_2+1)} \\
         \qquad +3 \sum_{i=1}^{n} \mathbb{I}(Y_i \in B^{J+1}_{(2t_1,2t_2+1)}) - \sum_{i=1}^{n} \mathbb{I}(Y_i \in B^{J+1}_{(2t_1+1,2t_2+1)}) & \\
        - \sum_{i=1}^{n} \mathbb{I}(Y_i \in B^{J+1}_{(2t_1,2t_2)})- \sum_{i=1}^{n} \mathbb{I}(Y_i \in B^{J+1}_{(2t_1+1,2t_2)}) & x \in B^{J+1}_{(2t_1+1,2t_2+1)} \\
        \qquad - \sum_{i=1}^{n} \mathbb{I}(Y_i \in B^{J+1}_{(2t_1,2i_t+1)}) +3 \sum_{i=1}^{n} \mathbb{I}(Y_i \in B^{J+1}_{(2t_1+1,2t_2+1)}) & \\
        \end{cases}
    \end{equation}
    Finally now using the inductive assumption (that $\hat{f}_{n,J} \in \mathcal{H}_{J}$ with parameter $\pmb{\pi}_J$), we have that for $x \in B^{J}_{(t_1,t_2)}$, $\hat{f}_{n,J}(x) = \frac{2^{2J}}{n} \sum_{t=1}^{n} \mathbb{I} (Y_t \in B_{(t_1,t_2)}^{J})$ and therefore using Equations \ref{nwbInductiveConnect}, \ref{nwbWaveletPart}, and that $B_{(t_1,t_2)}^{J} = \bigcup_{(z_1,z_2) \in S_{(t_1,t_2)} } B_{(z_1,z_2)}^{J+1}$, we have that for $x \in B^{J}_{(t_1,t_2)}$
    \begin{equation}
        \hat{f}_{n,J+1}(x) = 2^{2(J+1)} \begin{cases} \frac{\sum_{i=1}^{n} \mathbb{I}(Y_i \in B^{J+1}_{(2t_1,2t_2)})}{n} = {\pi}_{J+1,(2t_1+1,2t_2+1)} & x \in B^{J+1}_{(2t_1,2t_2)} \\
        \frac{\sum_{i=1}^{n} \mathbb{I}(Y_i \in B^{J+1}_{(2i_1+1,2i_2)})}{n} = {\pi}_{J+1,(2i_1+2,2i_2+1)} & x \in B^{J+1}_{(2i_1+1,2i_2)} \\
        \frac{\sum_{i=1}^{n} \mathbb{I}(Y_i \in B^{J+1}_{(2i_1+1,2i_2+2)})}{n} = {\pi}_{J+1,(2i_1,2i_2+1)}& x \in B^{J+1}_{(2i_1,2i_2+1)} \\
        \frac{\sum_{i=1}^{n} \mathbb{I}(Y_i \in B^{J+1}_{(2i_1+1,2i_2+1)})}{n} = {\pi}_{J+1,(2i_1+2,2i_2+2)} & x \in B^{J+1}_{(2i_1+1,2i_2+1)}
        \end{cases}
    \end{equation}
    This argument applies for all $(t_1,t_2) \in \{0,1,2,\dots,2^{J}-1\}^2$, and \linebreak $[0,1)^2 = \bigcup_{(t_1,t_2) \in \{0,1,2,\dots,2^{J}-1\}^2 } B^{J}_{(t_1,t_2)}$. Thus we conclude that $\hat{f}_{n,J+1} \in \mathcal{H}_{J+1}$ with parameter $\pmb{\pi}_{J+1}$. This concludes the inductive argument. \\
\end{proof}

Recall $\bar{P}_n \in \mathcal{H}_{\ceil{\log_2(k_n)}}$ with parameter $\bar{\pmb{\pi}}$ satisfying 
\[
\bar{\pi}_{(t_1+1,t_2+1)} \propto  \sum_{i=1}^{n} \mathbb{I}(Y_i \in B_{(t_1,t_2)}^{\ceil{\log_2(k_n)}})+\alpha_{(t_1+1,t_2+1),b_n},
\]
for $(t_1,t_2) \in \{0,1,2,\dots,2^{\ceil{\log_2(k_n)}}-1\}^2$. Thus referring now to the general $d$ dimensional setting, $\bar{P}_n$ is a dyadic histogram with order $k_n^d$ bins and by Lemma \ref{nwbLemma}, $\ddot{P}_n$ is a dyadic histogram with order $2^{J(n)d}$ bins.
From the requirement in the proof of \citet[Theorem~6]{niles2022minimax} that $n \asymp 2^{J(n)(d+s)}$ where $0 < s < 1$ is the assumed holder regularity of the density to be estimated, the requirement on $J(n)$ is that $J(n) \leq \log_2(n^{1/(d+s)})$. In particular the number of bins in $\ddot{P}_n$ is $\asymp n^{d/(d+s)}$. This contrasts with Theorem \ref{frequentistDwasser} of this paper, where for $d < 2v$, $k_n = n^{1/2v}$ and the number of bins in $\bar{P}_n$ is $\asymp n^{d/2v}$. Hence, in the $d < 2v$ case ($(d=1,v \geq 1), (d=2,v=2),(d=3,v=2)$), since $2v > d+s$ for $0 < s < 1$, there is a polynomial factor reduction in the number of bins needed in Theorem \ref{frequentistDwasser} compared to in \cite{niles2022minimax}.

\bibliographystyle{chicago}
\bibliography{npq,manyQntl}

\end{document}